\documentclass[a4paper, 11pt]{amsart}

\usepackage{amsmath, amssymb, amsthm, color, mathtools, stmaryrd, enumerate, hyperref, appendix, mathrsfs, bm, tikz, tikz-cd}
\usetikzlibrary{decorations.pathmorphing, decorations.markings}
\tikzset{blob/.style={circle, draw=black, fill=black, inner sep=0, minimum size=\blobsize}, snakeit/.style={decorate, decoration={snake, amplitude=0.25mm, segment length=4mm, post length=0.5mm}}}
\let\oldlrcorner=\lrcorner
\usepackage{mathabx}
\usepackage[capitalize]{cleveref}
\crefname{appsec}{Appendix}{Appendices}
\usepackage[justification=centering]{caption}

\let\originalleft\left
\let\originalright\right
\renewcommand{\left}{\mathopen{}\mathclose\bgroup\originalleft}
\renewcommand{\right}{\aftergroup\egroup\originalright}

\newcommand{\eps}{\varepsilon}
\renewcommand{\phi}{\varphi}
\newcommand{\m}{\mathfrak{m}}
\newcommand{\Z}{\mathbb{Z}}
\newcommand{\C}{\mathbb{C}}

\newcommand{\Po}{\mathfrak{P}}
\newcommand{\id}{\operatorname{id}}
\newcommand{\Id}{\operatorname{Id}}
\newcommand{\mf}{\mathrm{mf}}
\newcommand{\mffilt}{\mathrm{mf}^\mathrm{filt}}
\newcommand{\Cl}{\operatorname{\mathnormal{C\kern-0.15ex \ell}}}
\newcommand{\diff}{\mathrm{d}}
\newcommand{\rH}{\operatorname{H}}
\newcommand{\HH}{\operatorname{HH}}
\newcommand{\HF}{\operatorname{HF}}
\newcommand{\Fuk}{\mathcal{F}}

\newcommand{\Hom}{\operatorname{Hom}}

\newcommand{\mfE}{\mathscr{E}}
\newcommand{\Hess}{\operatorname{Hess}}
\newcommand{\LMF}[1]{\operatorname{\mathcal{LM}}^{#1}}
\newcommand{\F}{\Phi}
\newcommand{\Ext}{\operatorname{Ext}}
\newcommand{\vv}{\bm{\mathsf{v}}}
\newcommand{\ww}{\widecheck{\bm{\mathsf{w}}}}
\newcommand{\eend}{\operatorname{end}}
\newcommand{\End}{\operatorname{End}}
\newcommand{\muvv}{\mu_{\vv}}
\newcommand{\muzvv}{\mu_{\mathbf{0},\vv}}
\newcommand{\contract}{\mathbin{\oldlrcorner}}
\newcommand{\gr}{\operatorname{gr}}

\newcommand{\im}{\operatorname{im}}

\newcommand{\Jac}{\operatorname{Jac}}
\newcommand{\qprod}{\mathbin{\star}}
\newcommand{\calA}{\mathcal{A}}
\newcommand{\calB}{\mathcal{B}}

\makeatletter
\DeclareRobustCommand{\dgmod}[1]{\begingroup\newmcodes@\mathrm{dgmod-}#1\endgroup}
\makeatother

\theoremstyle{plain}
\newtheorem{thm}{Theorem}[section]
\newtheorem{lem}[thm]{Lemma}
\newtheorem{prop}[thm]{Proposition}
\newtheorem{cor}[thm]{Corollary}

\crefname{lem}{Lemma}{Lemmas}

\newtheorem{mthm}{Theorem}
\crefname{mthm}{Theorem}{Theorems}

\theoremstyle{remark}
\newtheorem{rmkx}[thm]{Remark}
\newenvironment{rmk}
  {\pushQED{\qed}\rmkx}
  {\popQED\endrmkx}
  
\newtheorem{exx}[thm]{Example}
\newenvironment{ex}
  {\pushQED{\qed}\exx}
  {\popQED\endexx}

  \theoremstyle{definition}
\newtheorem{defnx}[thm]{Definition}
\newenvironment{defn}
  {\pushQED{\qed}\defnx}
  {\popQED\enddefnx}
  
\newtheorem{warnx}[thm]{Warning}
\newenvironment{warn}
  {\pushQED{\qed}\warnx}
  {\popQED\endwarnx}

\title[$A_\infty$-deformations and mirror symmetry]{Superfiltered $A_\infty$-deformations of the exterior algebra,\\and local mirror symmetry}
\author{Jack Smith}
\address{St John's College, Cambridge, CB2 1TP, United Kingdom}
\email{j.smith@dpmms.cam.ac.uk}


\begin{document}

\begin{abstract}
The exterior algebra $E$ on a finite-rank free module $V$ carries a $\Z/2$-grading and an increasing filtration, and the $\Z/2$-graded filtered deformations of $E$ as an associative algebra are the familiar Clifford algebras, classified by quadratic forms on $V$.  We extend this result to $A_\infty$-algebra deformations $\calA$, showing that they are classified by formal functions on $V$.  The proof translates the problem into the language of matrix factorisations, using the localised mirror functor construction of Cho--Hong--Lau, and works over an arbitrary ground ring.  We also compute the Hochschild cohomology algebras of such $\calA$.

By applying these ideas to a related construction of Cho--Hong--Lau we prove a local form of homological mirror symmetry: the Floer $A_\infty$-algebra of a monotone Lagrangian torus is quasi-isomorphic to the endomorphism algebra of the expected matrix factorisation of its superpotential.
\end{abstract}

\maketitle


\section{Introduction}
\label{secIntroduction}


\subsection{Superfiltered deformations}
\label{sscFilteredDeformations}

Fix a ground ring $R_0$ (associative, commutative, unital) and a free $R_0$-module $V$ of rank $n$.  In this paper we study the exterior algebra $E^* \coloneqq \Lambda^* V$ over $R_0$, and certain deformations of its algebra structure.  There are many natural situations in which one starts with a $\Z$-graded object such as $E$ and then deforms it by adding correction terms of strictly lower degree to its structure maps, which preserve the reduced grading modulo $2$.  The resulting object is no longer $\Z$-graded, but is $\Z/2$-graded and filtered in a compatible way, and its associated graded is naturally identified with the original, undeformed object.  Such deformations of $E$, as an $A_\infty$-algebra, play an important role in mirror symmetry and our main result classifies them.

To make this precise, we introduce the following notion.

\begin{defn}
\label{defSuperfilteredMod}
A module $X$ is \emph{superfiltered} if it carries a $\Z/2$-grading $X = X^0 \oplus X^1$ and an increasing $\Z$-filtration $F^p X$ that is compatible with the grading in the sense that
\[
F^pX = (F^pX \cap X^0) \oplus (F^pX \cap X^1) \quad \text{and} \quad F^{p+1}X \cap X^p = F^pX \cap X^p
\]
for all $p \in \Z$ (here $X^p$ denotes the piece of degree $p$ mod $2$).  Given superfiltered modules $X$ and $Y$, we say a module map $f : X \to Y$ is \emph{superfiltered of degree $r \in \Z$} if it has degree $r$ mod $2$ and sends $F^pX$ into $F^{p+r}Y$.  We call the induced degree $r$ map $\gr X \to \gr Y$ its \emph{leading term}.
\end{defn}

All of our filtrations are assumed to be Hausdorff (the intersection of all filtered pieces is zero) and exhaustive (the union of all filtered pieces is the whole module).

\begin{ex}
\label{exSuperfiltrationFromGrading}
The prototype is a $\Z$-graded module $X^{(*)}$.  This is superfiltered by setting
\[
X^i = \bigoplus_{n \in i + 2\Z} X^{(n)} \quad \text{and} \quad F^pX = \bigoplus_{n \leq p} X^{(n)}.
\]
A superfiltered map $f$ of degree $r$ between two such modules can be written as a sum $f_r+f_{r-2}+f_{r-4}+\dots$, where each $f_i$ has degree $i$.  Its leading term is $f_r$.
\end{ex}

\begin{rmk}
A map which is superfiltered of degree $r$ is of course also superfiltered of degree $r'$ for all $r' \geq r$; the leading term depends on the choice of $r$ but in practice it should always be clear which choice we have in mind.
\end{rmk}

There are natural notions of superfiltered chain complexes, where the differential $\diff$ is superfiltered of degree $1$, and of superfiltered algebras, where the multiplication map and unit map (if relevant) are superfiltered of degree $0$.  Similarly there are superfiltered differential graded (dg-)algebras.  Note that superfiltered complexes are not filtered complexes in the standard sense: the differential maps \emph{up} one level in the filtration, rather than staying within a single filtered piece.

\begin{defn}
A \emph{superfiltered deformation} of a $\Z$-graded module $Y$ is a superfiltered module $X$ equipped with an isomorphism of graded modules $\gr X \cong Y$.  Superfiltered deformations of chain complexes or of (possibly unital) associative or dg-algebras are defined analogously.  Given two superfiltered deformations $X_1$ and $X_2$ of $Y$, a \emph{morphism} $f : X_1 \to X_2$ is a superfiltered map of modules (or chain complexes, algebras, etc) of degree $0$ such that $\gr f : \gr X_1 \to \gr X_2$ intertwines the identifications $\gr X_i \cong Y$.  Note that $\gr f$ is necessarily an isomorphism so if the filtrations of $X_1$ and $X_2$ are bounded (which is equivalent to $Y$ being bounded in degree) then $f$ itself must be an isomorphism.
\end{defn}

\begin{ex}
\label{exSplit}
Superfiltered deformations $X$ of the exterior algebra $E$ as an $R_0$-module are all trivial, i.e.~isomorphic to $E$ itself with its standard superfiltration as in \cref{exSuperfiltrationFromGrading}.  This is because the isomorphism $\gr X \cong E$ gives for each $p$ a short exact sequence
\[
0 \to F^{p-1}X \cap X^p \to F^p X \cap X^p \to E^p \to 0,
\]
and this can be split since $E^p = \Lambda^p V$ is free.  For $p = 0$ and $p = 1$ the sequences simplify to the isomorphisms $F^0 X \cong E^0 = R_0$ and $F^1 X \cap X^1 \cong E^1 = V$.
\end{ex}

\begin{ex}
\label{exClifford}
As a first non-trivial example, let us consider superfiltered deformations $A$ of $E$ as a unital associative algebra.  Given $v \in F^1A \cap A^1 \cong V$, the element $v^2$ lies in $F^2A \cap A^0$, and by the superfiltered condition its image in $\gr^2 A$ is $v \wedge v = 0$.  This means that $v^2$ is actually in $F^1A \cap A^0 = F^0A \cap A^0 \cong R_0$, so $v \mapsto v^2$ defines a quadratic form $Q : V \to R_0$.  Now let $TV$ denote the tensor algebra on $V$, and consider the canonical unital algebra homomorphism $TV \to A$ that extends the inclusion of $V$ into $A$ as $F^1 A \cap A^1$.  This factors through a  map $\pi : \Cl(Q) \to A$, where $\Cl(Q)$ is the Clifford algebra
\begin{equation}
\label{eqCliffordAlgebra}
\Cl(Q) \coloneqq TV / (v \otimes v - Q(v)).
\end{equation}
This Clifford algebra inherits a superfiltration from $TV$, and is naturally a superfiltered deformation of $E$.  The homomorphism $\pi : \Cl(Q) \to A$ is superfiltered of degree $0$, and $\gr \pi$ intertwines the identifications of $\gr \Cl(Q)$ and $\gr A$ with $E$, so $\pi$ is a morphism of superfiltered deformations.  In fact, it is an isomorphism of such, since the filtrations are bounded.
\end{ex}

The upshot of \cref{exClifford} is that associated to any superfiltered deformation $A$ of $E$ as a unital associative algebra is a quadratic form $Q$ on $V$, and $A$ is canonically isomorphic to $\Cl(Q)$ as a superfiltered deformation.  Our goal is to extend this understanding to the corresponding $A_\infty$-deformations of $E$.  There are some extra subtleties in the definitions in this case, so first we recap the standard $\Z$-graded theory.


\subsection{$A_\infty$-algebras and deformations}

Recall that a $\Z$-graded $A_\infty$-algebra $\calA$ over a ring $S$ comprises a $\Z$-graded $S$-module $\calA$ and for each $k \geq 1$ a degree $1$ operation
\[
\mu^k : \calA[1]^{\otimes_S k} \to \calA[1].
\]
These operations should satisfy
\begin{equation}
\label{eqAinfinityRelations}
\sum_{i ,j} (-1)^{\maltese_i} \mu^{k-j+1}(a_k, \dots, a_{i+j+1}, \mu^j(a_{i+j}, \dots, a_{i+1}), a_i, \dots, a_1) = 0
\end{equation}
for all $k$ and all homogeneous $a_1, \dots, a_k \in \calA$, where $\maltese_i = |a_1| + \dots + |a_i| - i$. As usual, $[1]$ denotes shift by $1$ (so $\calA[1]^i = \calA^{i+1}$), and summations without explicitly specified ranges mean `sum over all choices for which the expression makes sense'.  There are various different sign conventions in use but we follow Seidel \cite{SeidelBook}.

The first three relations say (up to the sign twist of \eqref{eqdgAinfinity}) that $\mu^1$ is a differential, $\mu^2$ satisfies a Leibniz rule, and $\mu^2$ is associative up to a homotopy given by $\mu^3$.  The $A_\infty$-structure is \emph{minimal} if $\mu^1=0$, and an element $1_\calA \in \calA^0$ is a \emph{strict unit} if
\[
\mu^2(a, 1_\calA) = a = (-1)^{|a|} \mu^2(1_\calA, a)
\]
for all $a$, and if $\mu^k(a_k, \dots, a_1)$ vanishes when $k \neq 2$ and some $a_i$ is equal to $1_\calA$.  The definitions go through with obvious modifications for $\Z/2$-gradings instead of $\Z$-gradings.

Dg-algebras correspond to those $A_\infty$-algebras with $\mu^k = 0$ for all $k \geq 3$.  Our convention is that the differential $\diff$ and product $\qprod$ are related to $\mu^1$ and $\mu^2$ via
\begin{equation}
\label{eqdgAinfinity}
\diff a = (-1)^{|a|}\mu^1(a) \quad \text{and} \quad a_2 \qprod a_1 = (-1)^{|a_1|}\mu^2(a_2, a_1).
\end{equation}
Specialising further, an associative algebra is an $A_\infty$-algebra with $\mu^k = 0$ for all $k \neq 2$, and a strict unit is then just a unit in the ordinary sense.  We can therefore view $E$ as a strictly unital, minimal $A_\infty$-algebra, and consider superfiltered deformations within this class.  There is, however, some choice in exactly what we mean by this.

To clarify things, recall that an $A_\infty$-algebra homomorphism $\F : \calA_1 \to \calA_2$ comprises a sequence of degree $0$ maps $\F^k : \calA_1[1]^{\otimes k} \to \calA_2[1]$ satisfying
\begin{multline}
\label{eqAinfinityHom}
\sum_{i,j} (-1)^{\maltese_i} \F^{k-j+1}(a_k, \dots, a_{i+j+1}, \mu_{\calA_1}^j(a_{i+j}, \dots, a_{i+1}), a_i, \dots, a_1)
\\ = \sum_r \sum_{\substack{s_1, \dots, s_r \\ s_1+\dots+s_r = k}} \mu_{\calA_2}^r(\F^{s_r}(a_k, \dots, a_{k-s_r+1}), \dots, \F^{s_1}(a_{s_1}, \dots, a_1)).
\end{multline}
A superfiltered deformation of a $\Z$-graded $A_\infty$-algebra $\calB$ could therefore mean a superfiltered $A_\infty$-algebra $\calA$ equipped with an $A_\infty$-algebra isomorphism $\Phi : \gr \calA \to \calB$.  We shall instead use the following slightly stronger definition.

\begin{defn}
\label{defAinftyDeformation}
A superfiltered deformation of a $\Z$-graded $A_\infty$-algebra $\calB$ is a superfiltered $A_\infty$-algebra $\calA$ equipped with an $A_\infty$-algebra isomorphism $\Phi : \gr \calA \to \calB$ such that $\Phi^k = 0$ for all $k \geq 2$.  Equivalently, it is a superfiltered $A_\infty$-algebra $\calA$ equipped with an identification $\Phi^1 : \gr \calA \to \calB$ of the underlying modules under which the operations on $\gr \calA$ coincide with those on $\calB$.
\end{defn}

This is in line with our motivating setup of taking a $\Z$-graded object and adding corrections to the structure maps of strictly lower degree.

\begin{ex}
\label{exBasicAinftyDef}
Take $\calA$ to have the same underlying module as $\calB$, with superfiltration induced by the $\Z$-grading.  Then equip it with $A_\infty$-operations of the form
\begin{equation}
\label{eqBasicDef}
\mu^k_\calA = \mu^k_\calB + \nu^k_{-1} + \nu^k_{-3} + \dots,
\end{equation}
where the $\mu^k_\calB$ are the operations on $\calB$ and the $\nu^k_i$ are maps $\calB[1]^{\otimes k} \to \calB[1]$ of degree $i$.  This is a superfiltered $A_\infty$-deformation of $\calB$ with $\Phi^1 = \id_\calB$.
\end{ex}

\begin{ex}
\label{exAinftyDef}
If $\calB$ is free of finite rank as a graded $R_0$-module then \cref{exBasicAinftyDef} is universal in the following sense.  Given any superfiltered $A_\infty$-deformation $\calA$ of $\calB$, we can lift the given identification $\Phi^1 : \gr \calA \to \calB$ to a superfiltered module isomorphism $\phi : \calA \to \calB$ as in \cref{exSplit}.  The requirement for $\calA$ to be a superfiltered $A_\infty$-deformation of $\calB$ is then that under the identification $\phi$ the operations on $\calA$ are of the form \eqref{eqBasicDef}.  This condition is independent of the choice of lift $\phi$ of $\Phi^1$.
\end{ex}

\begin{defn}
\label{defSuperfiltered}
A \emph{superfiltered $A_\infty$-deformation of $E$} is a superfiltered deformation $\calA$ of $E$ in the sense of \cref{defAinftyDeformation}, which is strictly unital and minimal.
\end{defn}

Before continuing, we record some easy properties of such an $\calA$ that we will use repeatedly.  It has an underlying unital associative algebra, whose product is given by $\qprod$ from \eqref{eqdgAinfinity} and whose unit is the strict unit $1_\calA$.  We'll sometimes denote this algebra by $\rH^*(\calA)$, to distinguish it from the full $A_\infty$-algebra $\calA$, even though the differential on $\calA$ is assumed to vanish.  The map $\Phi^1$ from \cref{defAinftyDeformation} is then, by the $k=2$ version of \eqref{eqAinfinityHom}, an algebra homomorphism $\gr \rH^*(\calA) \to E$.  A priori this homomorphism need not be unital, but unitality is forced by the fact that it's an isomorphism.  Consequently, the $A_\infty$-isomorphism $\Phi : \gr \calA \to E$ is automatically strictly unital, meaning that $\Phi^1(1_\calA) = 1_E$ and $\Phi^k(a_k, \dots, a_1) = 0$ when $k \geq 2$ and some $a_i$ is $1_\calA$.

\begin{rmk}
\label{rmkStrictlyUnital}
Recall that a dg- or $A_\infty$-algebra is cohomologically unital if its cohomology algebra is unital in the ordinary sense, and a dg- or $A_\infty$-homomorphism is cohomologically unital if the induced map on cohomology is unital.  It is a fact that any cohomologically unital $A_\infty$-algebra can be made strictly unital by a formal diffeomorphism given by explicit formulae which respect the grading and filtration, and so can any cohomologically unital $A_\infty$-map between strictly unital algebras; see the discussion in Seidel's book \cite[Section (2a)]{SeidelBook} or Lef\`evre-Hasegawa's thesis \cite[Sections 3.2.1--3.2.2]{LefevreThesis}.  This means that the distinction between cohomological and strict unitality is not fundamentally important.  Our restriction to the strictly unital case merely simplifies the exposition.
\end{rmk}

With \cref{defSuperfiltered} in place, it remains to define a morphism $\calA_1 \to \calA_2$ of superfiltered $A_\infty$-deformations of $E$.  This should clearly be a superfiltered map $\Psi : \calA_1 \to \calA_2$ of strictly unital $A_\infty$-algebras satisfying an extra condition, but there are two obvious choices for this condition, reflecting the two different perspectives of \cref{defAinftyDeformation}.  One could ask either that $\gr \Psi$ intertwines the two $A_\infty$-isomorphisms $\Phi_i : \gr \calA_i \to E$ or merely that $\gr \Psi^1$ intertwines the module identifications $\Phi_i^1 : \gr \calA_i \to E$.  In both cases, such morphisms are automatically invertible (one can write down an explicit inductive construction for the inverse), so we will call them equivalences.  We will refer to the two versions as \emph{$\infty$-equivalences} and \emph{$1$-equivalences}, since they correspond to the conditions $(\Phi_2 \circ \gr \Psi)^r = \Phi_1^r$ for $r=1, 2, \dots$ and $r = 1$ respectively.  It will be convenient for us to talk about the obvious intermediate notions so we make the following definition.

\begin{defn}
\label{defSuperfilteredEquivalence}
Fix superfiltered $A_\infty$-deformations $\calA_1$ and $\calA_2$ of $E$, with corresponding$A_\infty$- isomorphisms $\Phi_i : \gr \calA_i \to E$.  For $d \in \{1, 2, \dots, \infty\}$, a \emph{$d$-equivalence} $\calA_1 \to \calA_2$ is a superfiltered map $\Psi : \calA_1 \to \calA_2$ of strictly unital $A_\infty$-algebras satisfying $(\Phi_2 \circ \gr \Psi)^r = \Phi_1^r$ for $r = 1, \dots, d$.  A \emph{$0$-equivalence} is a superfiltered map $\Psi$ such that $\gr \Psi^1$ is an arbitrary isomorphism of graded modules $\gr \calA_1 \to \gr \calA_2$.  For all $d$, $d$-equivalence is easily checked to be an equivalence relation.
\end{defn}

\begin{rmk}
A $0$-equivalence is simply a $\Z/2$-graded $A_\infty$-isomorphism respecting the filtration.  A $1$-equivalence amounts to the same thing but with the additional condition that it respects the identifications $F^1\calA_i \cap \calA_i^1 \cong V$.
\end{rmk}

\begin{ex}
\label{exSuperfilteredEquivalence}
Suppose we lift each $\Phi_i^1$ to an identification of $\calA_i$ with $E$, as in \cref{exAinftyDef}.  Then a superfiltered $A_\infty$-algebra map $\Psi : \calA_1 \to \calA_2$ can be viewed as a sequence of maps $\Psi^r : E[1]^{\otimes r} \to E[1]$, each of which decomposes as
\[
\Psi^r_0 + \Psi^r_{-2} + \Psi^r_{-4} + \dots
\]
where $\Psi^r_i$ has degree $i$.  For $d \in \{1, 2, \dots, \infty\}$, $\Psi$ is a $d$-equivalence if and only if for $r=1, \dots, d$ the leading term $\Psi^r_0$ coincides with $\Id_E^r$, where $\Id_E$ is the identity $A_\infty$-automorphism of $E$ (i.e.~$\operatorname{Id}_E^1 = \id_E$ and $\operatorname{Id}_E^{\geq 2} = 0$).  It's a $0$-equivalence if and only if $\Psi^1_0$ is an arbitrary module automorphism of $E$.
\end{ex}


\subsection{Motivation}
\label{sscMotivation}

Dg-algebras $A$ are ubiquitous in homological algebra, as the endomorphism algebras of chain complexes $C$.  Passing to the cohomology $\rH^*(A)$ is convenient, but loses information in the sense that $A$ and $\rH^*(A)$ are not usually quasi-isomorphic as dg-algebras (taking the differential on $\rH^*(A)$ to be zero).  Under mild hypotheses, however, one can equip $\rH^*(A)$ with higher $A_\infty$-operations such that it becomes quasi-isomorphic to $A$ as an $A_\infty$-algebra.  An explicit construction, which we shall use later, was given by Markl \cite{Markl}, building on \cite{GugenheimStasheff,Merkulov,KontsevichSoibelman}.  The resulting $A_\infty$-algebra, with underlying module $\rH^*(A)$, is called a \emph{minimal model} for $A$.  Superfiltered deformations of $C$ frequently arise in nature, and lead to superfiltered deformations of $A$.  In good situations, it is reasonable to expect that these in turn induce superfiltered $A_\infty$-deformations of the minimal model.

This phenomenon is illustrated nicely in the study of \emph{matrix factorisations}, which are important objects in algebraic geometry and whose filtered versions are central to this paper.  We recall the precise definitions in \cref{secSurj}, but for now suppose we're given a regular function $w$ on a smooth variety $X$, and a coherent sheaf $\mathscr{F}$ on $X$ whose support is contained in $w^{-1}(0)$.  An algebraically-minded reader can think of $w$ as an element of a ring $R$, and $\mathscr{F}$ as a finitely-generated $R$-module of finite projective dimension that is annihilated by $w$.  Let $(C, \diff_1)$ be a finite projective resolution of $\mathscr{F}$, and $A$ its endomorphism algebra, so $\rH^*(A) = \Ext^*(\mathscr{F}, \mathscr{F})$.  Since multiplication by $w$ annihilates $\mathscr{F}$, it acts nullhomotopically on $C$, so there exists a degree $-1$ map $\diff_{-1} : C \to C$ such that $\diff_1 \diff_{-1} + \diff_{-1} \diff_1 = w \id_C$.  Then $(\diff_{-1})^2$ is a chain map of degree $-2$, and it is also nullhomotopic since $\Ext^{-2}(\mathscr{F}, \mathscr{F}) = 0$, so there exists a degree $-3$ map $d_{-3} : C \to C$ satisfying $\diff_1 + \diff_{-3} + \diff_{-3} \diff_1 = (\diff_{-1})^2$.  Continuing in this way, we build a superfiltered map
\[
\diff \coloneqq \diff_1 + \diff_{-1} + \diff_{-3} + \dots : C \to C
\]
satisfying $\diff^2 = w$.  The superfiltered complex $(C, \diff)$ is a matrix factorisation of $w$.  It is a superfiltered deformation of $(C, \diff_1)$, and its endomorphism dg-algebra is a superfiltered deformation of $A$.

We'll focus on the case where $X = V$, or more precisely a formal neighbourhood of the origin in $V$, and $\mathscr{F}$ is the skyscraper sheaf at the origin.  Algebraically, this means that $R$ is the ring of formal functions on $V$, and $\mathscr{F} = R/ \m$, where $\m$ is the kernel of the `evaluate at $0$' map.  In this case $\rH^*(A)$ is identified with $E$, and if $w$ is a formal function on $V$ with $w(0) = 0$, i.e.~$w \in \m$, then we can apply the above recipe to obtain a matrix factorisation $\mfE_0$.  We'll see later that if $0$ is actually a singular point of $w$, i.e.~$w \in \m^2$, then a minimal model for the endomorphism algebra of $\mfE_0$ is indeed a superfiltered $A_\infty$-deformation of $E$.  Our main result is essentially that $w$ can be recovered from this superfiltered $A_\infty$-deformation, and that \emph{every} superfiltered $A_\infty$-deformation arises in this way, up to $\infty$-equivalence.

Our interest in this classification problem originated in homological mirror symmetry, which predicts certain equivalences between Fukaya categories of symplectic manifolds and algebro-geometric categories of sheaves or matrix factorisations.  No knowledge of geometry is required to understand this paper, except for the short \cref{secMonotoneTori} at the end, but let us briefly summarise the relevant ideas to provide some context for the interested reader.  On the symplectic side of mirror symmetry, a particularly important role is played by the \emph{Floer $A_\infty$-algebras} of Lagrangian tori $L$.  These are obtained by adding `quantum corrections' to the operations on the ordinary cohomology algebra $\rH^*(L)$, which is simply the exterior algebra $E$ on $V = \rH^1(L)$ with its standard $A_\infty$-structure.  In general these quantum corrections respect the $\Z/2$-grading but may be very complicated.  However, a natural geometric hypothesis, \emph{monotonicity} of $L$, ensures that they are degree-decreasing and hence that the resulting Floer algebra is (modulo some technicalities) a superfiltered $A_\infty$-deformation of $E$.  This deformation is defined up to $1$-equivalence.  Under mirror symmetry a monotone torus is expected to correspond to a matrix factorisation arising from the skyscraper sheaf of a point $\rho$, as in the previous paragraph.  Proving mirror symmetry locally about $\rho$ therefore amounts to relating two superfiltered $A_\infty$-deformations of $E$, and this is what we do in \cref{secMonotoneTori}.


\subsection{Classification}

Before stating our main classification result, we need some notation.  As above, let $R$ denote the ring of functions on a formal neighbourhood of $0$ in $V$, or equivalently the ring of power series in elements of $V^\vee$.  Let $\m$ be the ideal of $R$ comprising those functions vanishing at the origin (this is the unique maximal ideal in $R$ if $R_0$ is a field).  To a superfiltered $A_\infty$-deformation $\calA$ of $E$, one can associate an element of $R$ as follows.

\begin{defn}
\label{defDiscPot}
The \emph{disc potential} $\Po \in R$ of $\calA$ is the formal function on $V$ given by
\[
v \mapsto \sum_{k \geq 2} \mu^k(v, \dots, v).
\]
Here we are viewing $v$ as an element of $\calA$ via the identification $F^1\calA \cap \calA^1 \cong V$.  A priori this function takes values in $F^2\calA \cap \calA^0$ but, similarly to the associative algebra case, the leading term is $\sum_{k \geq 2} \mu_E^k(v, \dots, v) = -v \wedge v = 0$ so the function actually lands in $F^0\calA = R_0$.

Since $\Po$ has no constant or linear terms, it is in fact an element of $\m^2 \subset R$.
\end{defn}

\begin{rmk}
This function was introduced by Fukaya--Oh--Ohta--Ono in \cite{FOOOBook}.  We insert the word `disc' into the name, following Sheridan \cite{SheridanFano}, to distinguish this from the plethora of other `potentials' one may encounter.
\end{rmk}

Another way to think of $\Po$, which will be useful later, is as follows.  Let $\vv \in \m \otimes V$ denote the germ of the Euler vector field.  If $v_1, \dots, v_n$ is a basis for $V$, and $x_1, \dots, x_n$ are the dual coordinates on $V$ so that $R = R_0 \llbracket x_1, \dots, x_n\rrbracket$ and $\m = (x_1, \dots, x_n)$, then $\vv = x_1v_1 + \dots + x_nv_n$.  We can then express $\Po$ by extending the $\mu^k$ $R$-multilinearly and setting
\[
\Po = \sum_{k \geq 2} \mu^k(\vv, \dots, \vv) = \sum_{k \geq 2} \sum_{i_1, \dots, i_k} \mu^k(v_{i_k}, \dots, v_{i_1}) x_{i_1}\cdots x_{i_k}.
\]

\begin{warn}
\label{warnPoConventions}
The disc potential can be defined for any choice of $A_\infty$-algebra sign conventions, but its value depends on this choice.  For example, some authors use operations $\widetilde{\mu}^k$ such that $\widetilde{\mu}^2$ induces the associative product $\qprod$ on cohomology without the sign $(-1)^{\lvert a_1 \rvert}$ appearing in \eqref{eqdgAinfinity}.  The quadratic parts of the disc potentials defined using the $\widetilde{\mu}^k$ and the $\mu^k$ then have opposite signs.
\end{warn}

\begin{defn}
For $d \in \{0, 1, 2, \dots, \infty\}$, formal functions $\Po_1$ and $\Po_2$ in $R$ are \emph{$d$-equivalent} if there exists an invertible formal change of variables $f : V \to V$ such that $\Po_1 = \Po_2 \circ f$, and such that $f = \id_V$ modulo $\m^{d+1}$.  Invertibility follows automatically from the latter if $d \geq 1$.
\end{defn}

Denoting $d$-equivalence of superfiltered $A_\infty$-deformations or of disc potentials by $\sim_d$, our main algebraic result is the following.

\begin{mthm}[\cref{propMainThm}]
\label{Theorem1}
For all $d \in \{0, 1, 2, \dots, \infty\}$, the map $\calA \mapsto \Po$ induces a bijection
\begin{equation}
\label{eqMainBijection}
\{\text{superfiltered $A_\infty$-deformations of $E$}\} \mathclose{}/\mathopen{} \sim_d \quad \longrightarrow \quad \m^2 \mathclose{}/\mathopen{} \sim_d.
\end{equation}
\end{mthm}

One could summarise this with the slogan: superfiltered $A_\infty$-deformations $\calA$ of $E$ are determined by the values of the (symmetrised) $A_\infty$-operations on filtration level $1$.  This is trivial when $n=1$, and in this case $\Po$ is in fact the full generating function of the $A_\infty$-operations not already prescribed by strict unitality.  When $n>1$, however, it is much less obvious.

\begin{rmk}
First order deformations of $E$ as a $\Z$-graded (or $\Z/2$-graded) $A_\infty$-algebra are classified by the Hochschild cohomology group $\HH^2(E)$ (respectively $\HH^\mathrm{even}(E)$).  The corresponding group for first order superfiltered deformations is $\HH^{\mathrm{even} < 2}(E)$.  The Hochschild--Kostant--Rosenberg Theorem \cite{HochschildKostantRosenberg} gives $\HH^*(E) \cong R \otimes E^*$, so first order superfiltered deformations are classified by $R \otimes E^0 = R$, and in fact the HKR map $\HH^0(E) \to R$ corresponds precisely to sending a deformation to its disc potential.  So \cref{Theorem1} says that first order deformations extend uniquely to genuine deformations modulo $\infty$-equivalence.  Related ideas are discussed in \cref{secObstructionDeformation}.
\end{rmk}

\begin{rmk}
\label{rmkGeneralisesClifford}
The classification of associative deformations in terms of quadratic forms can be seen as the truncation of this result to `degree $2$'.  Indeed, by definition of $\Po$ its quadratic part $-Q$ is such that the associative algebra underlying $\calA$ is $\Cl(Q)$.
\end{rmk}

\begin{rmk}
\label{rmkBeforeChangeOfVariables}
Suppose that $R_0$ is a field and that $A$ is an augmented $R_0$-algebra.  Under various hypotheses, it is known that from the $A_\infty$-structure maps $\mu^k : \Ext^1_A(R_0, R_0)^{\otimes k} \to \Ext^2_A(R_0, R_0)$ one can recover $A$, and hence the full $A_\infty$-algebra $\Ext_A^*(R_0, R_0)$; see for example Keller \cite{KellerAinfinity}, Lu--Palmieri--Wu--Zhang \cite{LuPalmieriWuZhang}, and Segal \cite{SegalAinfinity}.  In other words, $\Ext_A^*(R_0, R_0)$ is completely determined by the restriction of its operations to degree $1$.  This is very similar to our slogan for superfiltered deformations $\calA$, and the key step in our argument is essentially to realise $\calA$ as the $\Ext$-algebra of $R_0$ but in a matrix factorisation category.  I thank Ivan Smith for pointing out this connection.
\end{rmk}

It is not hard to see that a $d$-equivalence $\F : \calA_1 \to \calA_2$ induces a $d$-equivalence $\Po_1 = \Po_2 \circ f$ between the corresponding disc potentials: take $f$ to be the formal change of variables $f_\F$ given by
\[
f_\F = \sum_{k \geq 1} \F^k(\vv, \dots, \vv);
\]
if $\gr \F^k = \Id_E^k$ for $k \leq d$ then $f_\F = \vv + \sum_{k\geq d+1} \F^k(\vv, \dots, \vv)$, so $f_\F = \id_V$ modulo $\m^{d+1}$.  Thus \eqref{eqMainBijection} is well-defined, and the non-trivial task is to prove surjectivity and, more interestingly, injectivity.


\subsection{Idea of proof}
\label{sscOutline}

The strategy is as follows.  Given $w \in \m^2$ there is a $\Z/2$-graded dg-category $\mffilt(R, w)$ of filtered matrix factorisations of $w$.  This contains a distinguished (up to isomorphism) object $\mfE_0$, whose construction we sketched in \cref{sscMotivation}.  The endomorphism dg-algebra $\calB_0$ of $\mfE_0$ is a superfiltered deformation of the endomorphism algebra of the Koszul resolution of this module.  We apply Markl's recipe from \cite{Markl} to construct a minimal model $\calB_0^\mathrm{min}$ for $\calB_0$ and verify that it is a superfiltered $A_\infty$-deformation of $E$.  It comes equipped with an $A_\infty$-algebra quasi-isomorphism $\Pi : \calB_0 \to \calB_0^\mathrm{min}$.

The algebras $\calB_0$ and $\calB_0^\mathrm{min}$ were studied by Dyckerhoff, who showed \cite[Theorem 5.9]{Dyckerhoff} that the quadratic form defining the associated Clifford algebra $\rH^*(\calB_0)$ is the quadratic part of $-w$, and stated a formula relating the coefficients of the disc potential of $\calB_0^\mathrm{min}$ to the Taylor coefficients of $w$, up to sign.  In \cref{sscComputingPotential} we spell out the details of this computation and deduce

\begin{mthm}[\cref{propDiscPotential}]
\label{thmB0Potential}
The disc potential of $\calB_0^\mathrm{min}$ is $w$ itself.
\end{mthm}

This is essentially well-known---for example, it appears in \cite[Proposition 7.1]{SheridanCY} in characteristic zero and under the assumption (irrelevant for our purposes) that the quadratic part of $w$ vanishes---but we present it for completeness.

\begin{rmk}
The algebra $\calB_0^\mathrm{min}$ has also been studied in detail in recent work of Tu \cite{TuCategoricalSaito}, using ideas from deformation quantisation and Kontsevich formality.
\end{rmk}

\begin{warn}
There are two sources of sign ambiguity in \cref{thmB0Potential}.  One arises from the choice of $A_\infty$-algebra sign conventions as explained in \cref{warnPoConventions}.  The other arises from the identification of $\gr \calB_0^\mathrm{min}$, or equivalently of $\Ext_R(R/\m, R/\m)$, with $E$.  We use the identification that corresponds to resolving $R/\m$ by the Koszul complex $(R \otimes E, -\vv \wedge \bullet)$, with an element $e \in E$ acting on this resolution by $e \wedge \bullet$.  It may be considered more natural to use the differential $\vv \wedge \bullet$ instead, which would modify the identification $\Ext^i_R(R_0, R_0) \cong E^i$ by $(-1)^i$, and the $w = w(x)$ in \cref{thmB0Potential} would then become $w(-x)$.  The reason we keep the minus sign in the Koszul differential is so that it matches with a sign arising naturally in \cref{secmfE}.
\end{warn}

An immediate consequence of \cref{thmB0Potential} is

\begin{cor}
\label{corSurjective}
For all $d$ the map \eqref{eqMainBijection} is surjective.\hfill$\qed$
\end{cor}

\begin{rmk}
Dyckerhoff's formulae describe the $A_\infty$-operations $\mu_\mathrm{min}$ on $\calB_0^\mathrm{min}$ when restricted to classes in $V \cong F^1 \calB_0^\mathrm{min} \cap (\calB_0^\mathrm{min})^1$, and he states that this information does not determine the complete $A_\infty$-structure.  This is of course true in the sense that it doesn't lead easily to formulae for all operations (such formulae are given, when $R_0$ is a field of characteristic $0$, by Tu \cite[Section 3.4]{TuCategoricalSaito}), but one consequence of \cref{Theorem1} is that it \emph{does} determine the full $\infty$-equivalence class of $\calB_0^\mathrm{min}$, provided you also remember its filtration.
\end{rmk}

\begin{rmk}
There is a particular focus in the literature on the case where $w$ is a polynomial with an \emph{isolated} singularity at the origin, but we emphasise that none of our results have any such isolatedness hypothesis.  The only restriction on $w$ is that it lies in $\m^2$.
\end{rmk}

Now suppose $\calA$ is a given superfiltered $A_\infty$-deformation of $E$ with disc potential $\Po = w$.  Viewing this as an $A_\infty$-category with a single object we apply the localised mirror functor of Cho--Hong--Lau \cite{ChoHongLauLMF}.  This is a powerful tool for proving mirror symmetry and has been used with great success by these authors and others (see, e.g.~\cite{ChoHongLauNoncommutative, ChoHongLauTorus, ChoHongLauGluing}).  It is a variant of Fukaya's $A_\infty$-Yoneda embedding \cite{FukayaFloerHomologyFor3Manifold,FukayaFloerHomologyAndMirrorSymmetryII}, and some history and related constructions are discussed in the introduction to \cite{ChoHongLauLMF}.

In our application, the localised mirror functor provides an object $\mfE$ in $\mffilt(R, w)$ and an $A_\infty$-algebra homomorphism $\Phi : \calA \to \calB$, where $\calB$ is the endomorphism algebra of $\mfE$.  One could crudely describe its effect as converting the filtered $A_\infty$-problem into a filtered dg-problem in a systematic way, by repeatedly inserting $\vv$ into the $A_\infty$-operations to pull information from the higher operations down to the differential.  This unlocks the standard technique for attacking filtered complexes---spectral sequences---and we can construct a dg-algebra isomorphism $\Psi : \calB \to \calB_0$ by reducing everything to computations on the first page, which only involves the undeformed exterior algebra.

Composing our three maps we obtain an $A_\infty$-algebra homomorphism
\[
\calA \xrightarrow{\ \F\ } \calB \xrightarrow{\ \Psi\ } \calB_0 \xrightarrow{\ \Pi\ } \calB_0^\mathrm{min}
\]
between superfiltered $A_\infty$-deformations of $E$, and we show that it is actually an $\infty$-equivalence.  This proves injectivity in \cref{Theorem1} for $d=\infty$, and it is then straightforward to deduce injectivity for other $d$.  Note that, as hinted in \cref{sscMotivation} and \cref{rmkBeforeChangeOfVariables}, our proof actually provides a concrete representative of the equivalence class of $\calA$, generalising the description of superfiltered associative deformations of $E$ as Clifford algebras.  Namely
\begin{mthm}
If the disc potential $\Po$ of $\calA$ is equal to $w$ then $\calA$ is $\infty$-equivalent to $\calB_0^\mathrm{min}$.  So algebras of the form $\calB_0^\mathrm{min}$ are universal amongst superfiltered $A_\infty$-deformations of $E$.
\end{mthm}

In \cref{secObstructionDeformation} we discuss an alternative approach to proving \cref{Theorem1} using obstruction theory, which does not lead to such an explicit description of the algebras.  We also describe some related results in which the superfiltered hypothesis is replaced by other conditions.


\subsection{Hochschild cohomology}

To any $A_\infty$-algebra $\calA$ one can associate its \emph{Hochschild cohomology} $\HH^*(\calA)$, which is a graded-commutative unital associative algebra describing the self-$\Ext$s of $\calA$ as an $\calA$-$\calA$-bimodule.  It plays an important role in deformation theory, and also appears in TQFT and mirror symmetry as the closed sector of the open-string theory described by $\calA$.  In \cref{secHH} we prove the following.

\begin{mthm}[\cref{corHHA}]
\label{thmHHA}
If $\calA$ is a superfiltered $A_\infty$-deformation of $E$ with potential $\Po$ then there is a canonical isomorphism of unital $R_0$-algebras
\begin{equation}
\label{eqHHisom}
\HH^*(\calA) \cong \rH^*(\Cl(-\tfrac{1}{2}\Hess(\Po)), -\diff\Po \contract \bullet),
\end{equation}
where $\Cl$ is the Clifford algebra over $R$ on the module $V_R \coloneqq R \otimes_{R_0} V$, $\Hess(\Po)$ is the Hessian quadratic form of $\Po$ over $R$, and the differential $-\diff\Po \contract \bullet$ is contraction with $-\diff\Po$.
\end{mthm}

\begin{rmk}
We explain in \cref{sscIdentifyingHH} why $\tfrac{1}{2}\Hess(\Po)$ makes sense, even if $2$ is not invertible.
\end{rmk}

When $\Po = 0$ this reduces to the standard Hochschild--Kostant--Rosenberg (HKR) isomorphism \cite{HochschildKostantRosenberg} for the exterior algebra.  Turning on $\Po$ introduces an extra differential and deforms the product.

\begin{rmk}
Since Hochschild cohomology is graded-commutative, \cref{thmHHisom} implies that the right-hand side of \eqref{eqHHisom} is also graded-commutative, which is not obvious.  In \cref{thmHAv,propCentre} we realise it as a subalgebra of the centre of $\Cl(-\tfrac{1}{2}\Hess(\Po))\otimes_R \Jac(\Po)$, where $\Jac(\Po)$ is the Jacobian algebra $R/(\partial_i \Po)$.
\end{rmk}

Previous results in this direction have focused on the Hochschild cohomology of the category $\mf(R, \Po)$: we know that $\calA$ is $\infty$-equivalent to the endomorphism algebra of the matrix factorisation $\mfE_0$, so if $\mfE_0$ split-generates the category then we have $\HH^*(\mf(R, \Po)) \cong \HH^*(\calA)$ since $\HH^*$ is Morita-invariant.  In particular, Dyckerhoff \cite[Section 6.2]{Dyckerhoff} assumed that $\Po$ has isolated critical locus and computed $\HH^*(\mf(R, \Po)) \cong \Jac(\Po)$ by identifying $\HH^*$ with endomorphisms of the diagonal matrix factorisation in $\mf(R\otimes_{R_0} R, -\Po \otimes 1 + 1 \otimes \Po)$.  Segal \cite{SegalLGBModels} and C\u{a}ld\u{a}raru--Tu \cite{CaldararuTu} instead computed Hochschild (co)homology of $\mf(R, \Po)$ by viewing matrix factorisations as modules over the curved algebra $(R, \Po)$ and calculating certain Hochschild invariants for the latter.  They again obtained $\HH^*(\mf(R, \Po)) \cong \Jac(\Po)$ when the critical locus is isolated, and Segal suggested that in the non-isolated case the Hochschild homology should be
\[
\rH_*(R \otimes_{R_0} \Lambda V^\vee, -\diff\Po \wedge \bullet).
\]
Our complex $(\Cl(-\tfrac{1}{2}\Hess(\Po)), -\diff\Po \contract \bullet)$ is dual to Segal's complex $(R \otimes_{R_0} \Lambda V^\vee, -\diff\Po \wedge \bullet)$, non-canonically, so \cref{thmHHA} is consistent with his prediction.  However, since Hochschild homology has no product, the Hessian does not appear in his statement.


\subsection{Structure of the paper}

\Cref{secSurj} sets up the main algebraic objects---namely, filtered matrix factorisations---describes the factorisation $\mfE_0$, and then studies its endomorphism algebra $\calB_0$.  We construct the minimal model $\calB_0^\mathrm{min}$ and calculate its disc potential (\cref{thmB0Potential}).  In \cref{secmfE} we study the localised mirror functor in this setting, show that it gives a description of an arbitrary superfiltered deformation $\calA$ of $E$, and deduce our main classification result (\cref{Theorem1}).  We digress briefly in \cref{secObstructionDeformation} to discuss an alternative approach and some related results.  \Cref{secHH} then computes the Hochschild cohomology algebra of $\calA$ (\cref{thmHHA})---this doesn't use the earlier classification results, and can be read independently if desired.

Finally, in \cref{secMonotoneTori} we discuss the mirror symmetry picture outlined in \cref{sscMotivation}.  We use a geometric version of the localised mirror functor (also introduced by Cho--Hong--Lau, in \cite{ChoHongLauTorus}) to prove a local mirror symmetry result for monotone Lagrangian tori: the Floer $A_\infty$-algebra of such a torus is $1$-equivalent to a minimal model for the endomorphism algebra of the expected mirror matrix factorisation (\cref{thmLocalMS}).  We also deduce the folklore result that its disc potential is a suitable expansion of the superpotential $W_L$ (\cref{thmPoComputation}).  As mentioned above, only this last section assumes any knowledge of geometry.


\subsection{Acknowledgements}

I am grateful to Jonny Evans and Ed Segal for valuable feedback on an earlier draft, and more generally to Jonny, Ed and Yank\i~Lekili for many inspiring conversations over the last two years.  I am also indebted to the anonymous referees for a large number of helpful suggestions.  This paper was mostly written during a stay at The Fields Institute, as part of the Thematic Program on Homological Algebra of Mirror Symmetry, and I thank the Institute for financial support and for its hospitality and excellent working environment.  I am funded by EPSRC grant [EP/P02095X/2] and St John's College, Cambridge.


\section{The matrix factorisation $\mfE_0$}
\label{secSurj}

In this section we prove a spectral sequence lemma, review matrix factorisations, construct the object $\mfE_0$ from \cref{sscOutline}, and study its endomorphism algebra $\calB_0$ and the minimal model $\calB_0^\mathrm{min}$.  The section culminates with the computation of the disc potential of $\calB_0^\mathrm{min}$, proving \cref{thmB0Potential} and hence surjectivity of \eqref{eqMainBijection}.  Apart from our focus on filtrations, this is largely standard.


\subsection{A preliminary lemma}
\label{sscSetup}

In \cref{sscFilteredDeformations} we introduced the notions of superfiltered modules, algebras, and complexes.  Recall that a superfiltered complex is not a filtered complex in the usual sense since the differential maps one level up the filtration.  We would like to rectify this so we can use spectral sequences, and we do it as follows.

\begin{defn}
\label{defTcomplex}
Given a superfiltered complex $(C, \diff_C)$ over $R_0$, whose superfiltration is induced by a $\Z$-grading $C^{(n)}$, decompose the differential as $\diff_1 + \diff_{-1} + \diff_{-3} + \dots$, where $\diff_i$ has degree $i$ with respect to the $\Z$-grading.  Then define a new complex $(C^T, \diff_{C^T})$ by setting
\[
C^T = R_0[T^{\pm 1}] \otimes C \quad \text{and} \quad \diff_{C^T} = \diff_1 + T\diff_{-1} + T^2\diff_{-3} + \dots,
\]
extended $R_0[T^{\pm 1}]$-linearly; undecorated tensor products are always implictly taken over $R_0$.  This is $\Z$-graded, by placing $T^m \otimes C^{(n)}$ in degree $2m+n$, and is filtered in the usual sense by defining
\[
F^p C^T = T^p R_0[T] \otimes C.
\]
The induced spectral sequence has zeroth page $(C^T, \diff_1)$, with $T^m \otimes C^{(n)}$ lying in the $m$th column and $(m+n)$th row, and first page $R_0[T^{\pm 1}] \otimes \rH^*(C, \diff_1)$.  If $C^{(*)}$ is bounded in degree then the filtration on $C^T$ is finite in each degree, and the spectral sequence converges to $\rH^*(C^T, \diff_{C^T})$.
\end{defn}

If $C$ and $D$ are superfiltered complexes with superfiltrations induced by $\Z$-gradings, and $f$ is a superfiltered chain map $C \to D$ (of degree $0$), then there is an induced map $f^T : C^T \to D^T$ of $\Z$-graded filtered complexes, and hence an induced map on spectral sequences.  Later we shall use

\begin{lem}
\label{lemSSisom}
If both $C^{(*)}$ and $D^{(*)}$ are bounded in degree, and $f^T$ induces an isomorphism on some page of the spectral sequence, then $f$ is a quasi-isomorphism.
\end{lem}
\begin{proof}
Boundedness of $C^{(*)}$ and $D^{(*)}$ implies that the spectral sequences converge, and hence $f^T$ induces an isomorphism $\gr \rH^*(C^T) \to \gr \rH^*(D^T)$ between the limit pages.  Degreewise-finiteness of the filtration then means that $f^T$ is a quasi-isomorphism.  It is therefore enough to show that for each $j \in \Z$ there are isomorphisms $\rH^j(C) \to \rH^j(C^T)$ and $\rH^j(D) \to \rH^j(D^T)$ which intertwines the action of $f$ and $f^T$.  (Note the grading on $\rH^*(C)$ is mod $2$.)

The required isomorphism $\rH^j(C) \to \rH^j(C^T)$ can be described as being induced by the map
\[
i_j : (a_n) \in \bigoplus_{n \in 2\Z+j} C^{(n)} \mapsto \sum_{n \in 2\Z+j} T^{-\frac{(n-j)}{2}}a_n.
\]
This is a sort of chain map, in that it satisfies $d_{C^T} \circ i_j = i_{j+1} \circ d_C$, and this is enough for it to induce a map on cohomology.  To see that the map on cohomology is an isomorphism, observe that it is inverted by the chain map $C^T \to C$ given by setting $T=1$.  The construction for $D$ is analogous.
\end{proof}


\subsection{Filtered matrix factorisations}
\label{sscMatrixFactorisations}

We next collect the basic concepts of matrix factorisations; see Dyckerhoff \cite{Dyckerhoff} (whose treatment we follow), or the originating paper of Eisenbud \cite{EisenbudCompleteIntersection}, for a much fuller discussion.  We also define the obvious filtered modifications.

\begin{defn}
Given a ring $R$ and an element $w \in R$, a \emph{matrix factorisation} of $w$ over $R$ comprises a $\Z/2$-graded $R$-module $X = X^0 \oplus X^1$, which is finitely-generated and projective in each degree, equipped with an $R$-linear endomorphism $\diff$ of degree $1$ such that $\diff^2 = w\id_X$.  These form the objects of a $\Z/2$-graded dg-category $\mf(R, w)$ over $R$, in which $\hom^i (X, X')$ comprises $R$-linear maps $X \to X'$ of degree $i$, with differential
\[
\diff f = \diff_{X'} \circ f - (-1)^{|f|} f \circ \diff_X.
\]
Composition is defined in the obvious way.
\end{defn}

\begin{rmk}
Unfortunately there seems to be no more standard name for the endomorphism $\diff$ of $X$ than the generic term \emph{twisted differential}.  We propose (without hope or expectation that it will catch on) the name \emph{squifferential}, both because it is suggestive of `a squiffy differential', and because it is the \emph{squ}are of $\diff$ which is equal to the element $w$ that we're factorising.
\end{rmk}

Next we introduce the filtered versions we need.

\begin{defn}
A \emph{filtered matrix factorisation} is a matrix factorisation $(X, \diff)$ such that $X$ is superfiltered, and $\diff$ is superfiltered of degree $1$.  These form a superfiltered dg-category over $R$, meaning that morphism spaces are superfiltered complexes and composition is superfiltered of degree $0$, which we denote by $\mffilt(R, w)$.
\end{defn}


\subsection{Defining $\mfE_0$}
\label{sscDefE0}

Now restrict to the case where $R$ is the ring of formal functions on $V$ as in \cref{secIntroduction}, and $w$ lies in the ideal $\m^2$.  Recall that we denote by $v_1, \dots, v_n$ a basis for $V$, and by $x_1, \dots, x_n$ the dual coordinates on $V$, so that $R = R_0 \llbracket x_1, \dots, x_n \rrbracket$ and $\m = (x_1, \dots, x_n)$.  Since $w$ is in $\m^2$ we can pick $w_i \in \m$ such that $w = \sum_i x_iw_i$.  In other words, $w = \ww \contract \vv$, where $\ww = \sum_i w_i v_i^\vee \in \m \otimes V^\vee$, $\vv = \sum_i x_i v_i$, and $\contract$ denotes (the $R$-bilinear extension of) contraction between $V^\vee$ and $\Lambda V$.  Here $v_1^\vee, \dots, v_n^\vee$ is the dual basis for $V^\vee$.  Note that $x_i$ and $v_i^\vee$ are in some sense the same thing, but we think of the former as an element of $R$ and the latter as an element of $V^\vee$.

We define the filtered matrix factorisation $\mfE_0$ as follows (this is simply the stabilisation of $R_0$ in the sense of \cite[Section 2.3]{Dyckerhoff}).  Take the underlying module to be $E_R \coloneqq R \otimes E$, where $E = \Lambda V$ as usual.  Equivalently we could set $E_R = \Lambda V_R$, where $V_R \coloneqq R \otimes V$ and the exterior algebra is taken over $R$.  We equip this module with the superfiltration induced by its $\Z$-grading, and with the squifferential
\[
\diff_{\mfE_0} : a \mapsto -(\vv \wedge a + \ww \contract a).
\]
This is indeed superfiltered of degree $1$ and squares to $w \id_{E_R}$.  The reason for the overall minus sign is to make the leading term agree with that of the matrix factorisation $\mfE$ that we define later.

Let $\calB_0$ denote the superfiltered dg-algebra $\eend_{\mffilt}(\mfE_0)$.  Its underlying $R$-algebra is $\End_R(E_R)$, and the differential is graded-commutator with $\diff_{\mfE_0}$, i.e.
\[
\diff_{\calB_0}f = [\diff_{\mfE_0}, f].
\]
We can view it as an $A_\infty$-algebra via \eqref{eqdgAinfinity}, and then
\[
\mu^1_{\calB_0}(f) = -[f, \diff_{\mfE_0}].
\]


\subsection{The algebra $E_\mathrm{dg}$}
\label{sscEdg}

Before studying $\calB_0$ we focus on its leading order part.  We denote this dg-algebra by $E_\mathrm{dg}$ since it will turn out to be canonically quasi-isomorphic to $E$.  Precisely:

\begin{defn}
$E_\mathrm{dg}$ is the $\Z$-graded dg-algebra over $R$ with underlying $R$-algebra $\End_R(E_R)$ and with differential
\[
\diff_{E_\mathrm{dg}}f = [-\vv \wedge \bullet, f] \quad \text{or equivalently} \quad \mu^1_{E_\mathrm{dg}}(f) = [f, \vv \wedge \bullet].\qedhere
\]
\end{defn}

Given a finite rank free module $W$ over a ring $S$, it is well-known that $\End_S(\Lambda W)$ is isomorphic as an algebra to the Clifford algebra $\Cl (W \oplus W^\vee, Q_\mathrm{taut})$, where $Q_\mathrm{taut}$ is the tautological quadratic form defined by
\[
Q_\mathrm{taut}(w+\theta) = \theta(w)
\]
for $w \in W$ and $\theta \in W^\vee$.  Concretely, it is generated by endomorphisms of the form `wedge with $w$' and `contract with $\theta$', whose graded-commutator is $\theta(w) \id_{\Lambda W}$.  This Clifford algebra can be canonically identified with $(\Lambda W) \otimes (\Lambda W^\vee)$ by moving all of the $W$ terms to the left and $W^\vee$ terms to the right, and as such it carries a bigrading in which the piece of bidgree $(q, r)$ is $(\Lambda^q W) \otimes (\Lambda^r W^\vee)$.  However, the bigrading is not respected by the product.  The overall grading is given by $q-r$.

\begin{rmk}
In the work of Dyckerhoff \cite{Dyckerhoff} and Sheridan \cite[Section 7]{SheridanCY} this Clifford algebra appears as an algebra of differential operators in odd supercommuting variables.
\end{rmk}

Applying this in the case $S=R$ and $W=V_R$ we obtain a canonical identification of $E_\mathrm{dg}$ with $\Cl(V_R \oplus V_R^\vee, Q_\mathrm{taut})$, which we shall use freely from now on.  Under this identification our expression for $\mu^1_{E_\mathrm{dg}}$ simplifies to
\[
\mu^1_{E_\mathrm{dg}}(f) = [f, \vv].
\]
Although the product does not respect the bigrading, $\mu^1_{E_\mathrm{dg}}$ does, and has bidegree $(0, -1)$.  We in fact have a \emph{trigrading} in this case, in which the piece of tridegree $(p, q, r)$ is $R_p \otimes (\Lambda^q V) \otimes (\Lambda^r V^\vee)$, and with respect to this $\mu^{1}_{E_\mathrm{dg}}$ has tridegree $(1, 0, -1)$.  Here $R_p$ denotes the space of homogeneous polynomials of total degree $p$ in the $x_i$; note that this is consistent with the existing meaning of $R_0$.

There is an obvious $R_0$-algebra map $\iota_0 : E \to E_\mathrm{dg}$, of degree $0$, which simply includes $E^*$ as the piece of tridegree $(0, *, 0)$.  There is also an obvious map $\pi : E_\mathrm{dg} \to E$, which projects out the other trigraded pieces.  Another way to think of $\pi$ is that it sends $f \in \End_R(E_R)$ to the reduction of $f(1)$ modulo $\m$.  We then have a splitting
\[
E_\mathrm{dg} = \im \iota_0 \oplus \ker \pi,
\]
and the differential vanishes on the first summand and preserves the second.  The main result of this subsection is

\begin{lem}
\label{lemetam1}
The map $\iota_0$ induces an $R_0$-algebra isomorphism $E \to \rH^*(E_\mathrm{dg})$.  Moreover, the complex $\ker \pi$ is contractible, i.e.~there exists an $R_0$-linear chain homotopy
\[
\eta_{-1} : \ker \pi \to \ker \pi
\]
of degree $-1$ satisfying
\[
\mu^1_{E_\mathrm{dg}} \circ \eta_{-1} + \eta_{-1} \circ \mu^1_{E_\mathrm{dg}} = -\id_{\ker \pi}.
\]
(The reason for the minus sign on the right-hand side will become clear in \cref{lemeta}.)
\end{lem}
\begin{proof}
We can think of $E_\mathrm{dg}$ as the endomorphism algebra of the Koszul complex $(E_R, -\vv \wedge \bullet)$ resolving $E^n_R/\m \cong \Lambda^n V \cong R_0[-n]$ as an $R$-module, where $E^n_R$ denotes the degree $n$ part of $E_R$.  Therefore $\rH^*(E_{\mathrm{dg}})$ is the $\Ext_R$-algebra of $R_0[-n]$, or equivalently of $R_0$.  By computing this in terms of maps $(E_R, -\vv \wedge \bullet) \to R_0[-n]$, we see that it is isomorphic to $E$ itself as a $\Z$-graded $R$-algebra, with $e \in E$ acting on $(E_R, -\vv \wedge \bullet)$ as $x \mapsto e \wedge x$.  Hence $\iota_0$ induces an isomorphism on cohomology, and it remains to show $\ker \pi$ is contractible.

Recall that a bounded complex of projective modules is contractible if and only if it's acyclic.  By the previous paragraph we know that $\ker \pi$ is acyclic, and it's obviously bounded.  However, its graded pieces or not in general projective over $R_0$ (apart from the trivial case where the module $V$ is zero) since they involve countably infinite products of $R_0$.  To rectify this, we split $\ker \pi$ into free subcomplexes as follows.

For $s, q \geq 0$, let $K^{s, q}$ be the sum of the pieces of tridegree $(p, q, r)$ where $p+r = s$, i.e.
\[
K^{s, q} = \bigoplus_{\substack{p, r \\ p+r = s}} R_p \otimes (\Lambda^q V) \otimes (\Lambda^r V^\vee).
\]
We then have
\[
E_\mathrm{dg} = \prod_{s=0}^\infty \bigoplus_{q=0}^n K^{s, q} \quad \text{and} \quad \ker \pi = \prod_{s=1}^\infty \bigoplus_{q=0}^n K^{s,q}
\]
(note the different $s$ limits), and since $\mu^1_{E_\mathrm{dg}}$ has tridegree $(1, 0, -1)$ it respects these splittings.  Each $K^{s,q}$ is a bounded complex of free $R_0$-modules, and is acyclic for $s \geq 1$ because $\ker \pi$ is.  Thus each $K^{s,q}$ for $s \geq 1$ is contractible, and combining the individual homotopies gives the required $\eta_{-1}$.
\end{proof}


\subsection{The algebra $\calB_0$}
\label{sscEndmfE0}

We now return to studying $\calB_0 = \eend_{\mffilt}(\mfE_0)$.  As with $E_\mathrm{dg}$ the underlying algebra is naturally identified with $\Cl(V_R \oplus V_R^\vee, Q_\mathrm{taut})$, but now we have
\[
\mu^1_{\calB_0}(f) = [f, \vv + \ww].
\]
Recall that $\ww = \sum_i w_i v_i^\vee$, where the $w_i \in \m$ were chosen so that $\sum_i x_i w_i = w$.  Beware that this new $\mu^1$ does not respect the bigrading: the $\vv$ and $\ww$ terms have bidegree $(0, -1)$ and $(-1, 0)$ respectively.

Our goal in this subsection is to show that the cohomology $\rH^*(\calB_0)$ is isomorphic to the exterior algebra $E$ as a superfiltered module (canonically at the associated graded level), and verify the hypotheses that allow us to transfer the dg-structure to an $A_\infty$-structure on $E$ to give our minimal model $\calB_0^\mathrm{min}$.  We will do this by modifying the maps $\iota_0$ and $\eta_{-1}$ from \cref{sscEdg} to give a chain map $\iota : E \to \calB_0$ and a nullhomotopy $\eta : \ker \pi \to \ker \pi$ of $\ker \pi$ equipped with $\mu^1_{\calB_0}$.  (Note that $\ker \pi$ is automatically a subcomplex of $\calB_0$, because $\mu^1_{\calB_0}$ vanishes modulo $\m$, i.e.~lands in the pieces of tridegree $(p, q, r)$ satisfying $p \geq 1$.)

\begin{rmk}
Although $\calB_0$ is defined over $R$, and the isomorphism $\rH^*(\calB_0) \cong E$ holds as $R$-modules, in order to transfer the $A_\infty$-structure from $\calB_0$ to $E$ we will have to work over $R_0$ instead.  This is not a problem because ultimately it's $A_\infty$-algebras over $R_0$ that we care about.
\end{rmk}

The first task is to define $\iota$ on $E^1 = V$, i.e.~to correct the $v_i$ in $\calB_0$ to cocycles.  We can do this explicitly by hand, as follows.  Since each $w_i$ lies in the ideal $\m$ we have
\[
w_i = -\sum_j \lambda_{ij} x_j
\]
for some $\lambda_{ij}$ in $R$.

\begin{lem}
For each $i$, the element $f_i$ in $\calB_0$ given by $v_i + \sum_j \lambda_{ij}v_j^\vee$ is a cocycle.
\end{lem}
\begin{proof}
For each $i$ we have
\[
\mu^1_{\calB_0}(f_i) = \big[v_i + \sum_j \lambda_{ij}v_j^\vee, \vv + \ww\big] = [v_i, \ww] + \sum_j \lambda_{ij}[v_j^\vee, \vv] = w_i + \sum_j \lambda_{ij}x_j,
\]
which vanishes by construction.
\end{proof}

We now define the full map $\iota : E \to \calB_0$.  One natural thing to try is to send each basis element $v_{i_1} \wedge \dots \wedge v_{i_k}$ of $E$, with $i_1 < \dots < i_k$, to the product $f_{i_1} \dots f_{i_k}$ in $\calB_0$.  This is indeed a chain map, because $f_{i_1} \dots f_{i_k}$ is a cocycle by the Leibniz rule, and it's superfiltered of degree $0$.  It does not necessarily satisfy $\pi \circ \iota = \id_E$ though, which will be useful later.  However, it is straightforward to fix this by projecting onto appropriate bigraded pieces, as we now explain.

Each $f_i$ splits into pieces of bidegree $(1, 0)$ and $(0, 1)$, namely $v_i$ and $\sum_j \lambda_{ij} v_j^\vee$ respectively, so $f_{i_1} \dots f_{i_k}$ may obviously have contributions in bidegree $(k, 0), (k-1, 1), \dots, (0, k)$, i.e.~in bidegree $(k-a, a)$ for $a=0, \dots, k$.  Less obviously, it may also have contributions in bidegree $(k-a-b, a-b)$ for $b > 0$, coming from the Clifford relations in $\Cl(V_R \oplus V_R, Q_\mathrm{taut})$.  Let $\iota(v_{i_1} \wedge \dots \wedge v_{i_k})$ be the result of projecting out these non-obvious contributions.  In other words, take $f_{i_1} \dots f_{i_k}$ and project onto the pieces whose bidegree $(q,r)$ satisfies $q+r = k$.  This is still a cocycle because $\mu^1_{\calB_0}$ decreases $q+r$ by $1$, so an element of $\calB_0$ is a cocycle if and only if each of its pieces of fixed $q+r$ is also a cocycle.

The resulting map $\iota$ is then an $R_0$-linear chain map, which is superfiltered of degree $0$.  Moreover, its leading term corresponds to the bidegree $(k, 0)$ part in the previous paragraph, which is precisely $\iota_0$.  Beware, however, that $\iota$ is not in general an algebra map since the $f_i$ need not graded-commute.

\begin{lem}
\label{lemiotaqis}
The map $\iota : (E, 0) \to (\calB_0, \mu^1_{\calB_0})$ is a quasi-isomorphism of chain complexes over $R_0$.
\end{lem}
\begin{proof}
We already saw that $\iota$ is a chain map.  It remains to check that it induces an isomorphism on cohomology, and by \cref{lemSSisom} it suffices to show that $\iota^T$ induces an isomorphism on some page of the associated spectral sequences for $E^T$ and $\calB_0^T$.  It actually induces an isomorphism on the \emph{first} page, because the leading term $\iota_0$ gives a quasi-isomorphism from $E$ to $E_\mathrm{dg}$.
\end{proof}

To transfer the dg-structure on $\calB_0$ to an $A_\infty$-structure on $E$, to give our minimal model $\calB_0^\mathrm{min}$, we use Markl's construction from \cite{Markl}.  For this we need to show that we are in his `Situation 1'.  Concretely this means we need chain maps $\iota : E \to \calB_0$ and $\pi : \calB_0 \to E$ and a chain homotopy from $\id_{\calB_0}$ to $\iota \circ \pi$.  We already have the chain maps $\iota$ and $\pi$ (in fact our $\pi$ is precisely the $\Psi$ of \cite[Lemma 7.19]{ChoHongLauLMF}), which are superfiltered of degree $0$, and which additionally satisfy $\pi \circ \iota = \id_E$.  The main result of this subsection is

\begin{lem}
\label{lemeta}
There exists an $R_0$-linear chain homotopy $\eta$ satisfying
\begin{equation}
\label{eqetaCondition}
\mu^1_{\calB_0} \circ \eta + \eta \circ \mu^1_{\calB_0} = \iota \circ \pi - \id_{\calB_0}.
\end{equation}
Moreover $\eta$ can be chosen to be superfiltered of degree $-1$, and to satisfy the \emph{side conditions} $\eta \circ \iota = 0$, $\pi \circ \eta = 0$, and $\eta^2 = 0$.
\end{lem}
\begin{proof}
First we decompose $\calB_0$ as the direct sum of the subcomplexes $\im \iota$ and $\ker \pi$ over $R_0$.  Since $\pi$ and $\iota$ are superfiltered of degree $0$, we can use them to split each grading- and filtration-level, and hence this decomposition is actually as a direct sum of \emph{superfiltered} subcomplexes.  It thus suffices to construct $\eta$ on each summand.

We define $\eta$ to be zero on the $\im \iota$ summand.  This satisfies \eqref{eqetaCondition} using the fact that $\pi \circ \iota = \id_E$, and the side conditions are automatic here.  We are left to focus on $\ker \pi$, on which we want
\begin{equation}
\label{eqkerpihtpy}
\mu^1_{\calB_0} \circ \eta + \eta \circ \mu^1_{\calB_0} = -\id_{\ker \pi}.
\end{equation}
For the rest of the proof we denote the restriction $\mu^1_{\calB_0}|_{\ker \pi}$ by $\diff$, and drop the $\circ$ symbols for brevity.

Before getting into the construction, note that $\ker \pi$ inherits from $\calB_0 = \End_R(E_R)$ a $\Z$-grading, and this is what induces its superfiltration.  The map $\diff$ decomposes as $\vv+\ww$, and we denote the two pieces by $\diff_1$ and $\diff_{-1}$, of degree $1$ and $-1$ respectively.  Both of these pieces preserve $\ker \pi$ since they vanish modulo $\m$.  With respect to this $\Z$-grading the desired map $\eta$ decomposes as $\eta_{-1}+\eta_{-3}+\dots$ and \eqref{eqkerpihtpy} splits into graded pieces
\begin{equation}
\label{eqeta}
\diff_1 \eta_{-2i-1} + \diff_{-1} \eta_{-2i+1} + \eta_{-2i-1} \diff_1 + \eta_{-2i+1} \diff_{-1} = \begin{cases} -\id_{\ker \pi} & \text{if } i=0 \text{ (where $\eta_1 \coloneqq 0$)} \\ 0 & \text{otherwise.}\end{cases}
\end{equation}
The map $\diff_1$ is precisely $\mu^1_{E_\mathrm{dg}}$ restricted to $\ker \pi$, so $\eta_{-1}$ constructed in \cref{lemetam1} satisfies the $i=0$ case.  We shall construct $\eta_{-3}, \eta_{-5}, \dots$ inductively, and then deal with the side conditions at the end.

Suppose then that for some $j \geq 1$ we have built $\eta_{-1}, \dots, \eta_{-2j+1}$, of the correct degrees, satisfying \eqref{eqeta} for $i=0, \dots, j-1$.  Let $\theta = \diff_{-1} \eta_{-2j+1} + \eta_{-2j+1} \diff_{-1}$ and define $\eta_{-2j-1}$ to be $\theta \eta_{-1}$.  This has degree $-2j-1$, and the left-hand side of the $i=j$ case of \eqref{eqeta} is
\[
\diff_1 \theta \eta_{-1} + \theta \eta_{-1} \diff_1 + \theta = \theta (\diff_1 \eta_{-1}+\eta_{-1} \diff_1) + \theta = 0
\]
(the first equality uses the fact that $\diff_1 \theta = \theta \diff_1$, obtained by taking the commutator of the $i=j-1$ case of \eqref{eqeta} with $\diff_1$).  We thus obtain all components of $\eta$ by induction; the process terminates after finitely many steps since $\calB_0$ is bounded in degree.

Finally we deal with the side conditions.  By construction, $\eta$ vanishes on $\im \iota$ and lands in $\ker \pi$, so the only non-obvious condition is $\eta^2 = 0$.  This may not hold for $\eta$ as defined, but we can remedy this by replacing $\eta$ on $\ker \pi$ with $\eta - \diff \eta^3 \diff$.  This doesn't affect any of the other properties.  To see that $\eta - \diff \eta^3 \diff$ does indeed square to $0$, note that the homotopy condition $\diff \eta + \eta \diff = -\id_{\ker\pi}$ gives $\diff \eta^2 = \eta^2 \diff$ (after multiplying it on the left by $\eta$ and then, separately, on the right) and also $\diff \eta \diff = - \diff$ (after multiplying it on the left or right by $\diff$).  We then have on $\ker \pi$ that
\begin{align*}
(\eta - d\eta^3\diff)^2 &= \eta^2 - \diff \eta^3 \diff \eta - \eta \diff \eta^3 \diff = \eta^2 - \eta^2 \diff \eta \diff \eta - \eta \diff \eta \diff \eta^2
\\ &= \eta^2 + \eta^2 \diff \eta + \eta \diff \eta^2 = \eta(\id_C + \eta \diff + \diff \eta)\eta = 0.\qedhere
\end{align*}
\end{proof}

Markl's construction in \cite{Markl} then gives

\begin{prop}
There exists an $A_\infty$-algebra $\calB_0^\mathrm{min}$ with underlying module $E$, and a homomorphism $\Pi : \calB_0 \to \calB_0^\mathrm{min}$ of $A_\infty$-algebras over $R_0$, which extends $\pi$ in the sense that $\Pi^1 = \pi$.\hfill$\qed$
\end{prop}


\subsection{Properties of the minimal model}

We need to establish some basic properties of the $A_\infty$-algebra $\calB_0^\mathrm{min}$ and the $A_\infty$-map $\Pi$.  These are given by the following three lemmas.

\begin{lem}
\label{lemB0min}
$\calB_0^\mathrm{min}$ is a superfiltered $A_\infty$-deformation of $E$.
\end{lem}
\begin{proof}
We need to show that the $A_\infty$-operations $\mu^k_\mathrm{min}$ on $\calB_0^\mathrm{min}$ are superfiltered and strictly unital, and that they reduce to the standard $A_\infty$-structure on $E$ at the associated graded level.

The formula for $\mu^k_\mathrm{min}$ given by Markl \cite[Equation (1)]{Markl} is of the form $\pi \circ \mathbf{p}_k \circ \iota^{\otimes k}$, where the `kernel' $\mathbf{p}_k : \calB_0[1]^{\otimes k} \to \calB_0[1]$ of degree $1$ is defined by $\mathbf{p}_2 = \mu_{\calB_0}^2$ and then inductively by
\begin{align*}
\mathbf{p}_k (a_k, \dots, a_1) &= \sum_{r\geq 2} \sum_{\substack{s_1, \dots, s_r \\ s_1+\dots+s_r = k}} \mu^r_{\calB_0} (\eta\circ \mathbf{p}_{s_r} (a_k, \dots, a_{k-s_r+1}), \dots, \eta\circ \mathbf{p}_{s_1} (a_{s_1}, \dots, a_1))
\\ &= \sum_{j=1}^{k-1} \mu^2_{\calB_0}(\eta\circ \mathbf{p}_{k-j} (a_k, \dots, a_{j+1}), \eta\circ \mathbf{p}_{j} (a_j, \dots, a_1))
\end{align*}
Here $\eta \circ \mathbf{p}_1$ is interpreted as $\id_{\calB_0}$; we don't need $\mathbf{p}_1$ itself since $\mu^1_\mathrm{min}$ is already chosen to be zero.  Note that Markl's formulae contain non-trivial signs, whereas ours do not.  The reason for the difference is that we are using Seidel's sign conventions, which differ from Markl's.  The signs for Seidel's conventions appear in \cite[Equation (1.18)]{SeidelBook}: his $\mu^k_\calA$, $\mu^k_\calB$, $\mathcal{G}^1$, $\mathcal{F}^1$, $T^1$, and $\mathcal{F}^{k > 1}$ correspond to our $\mu^k_\mathrm{min}$, $\mu^k_{\calB_0}$, $\pi$, $\iota$, $\eta$, and $\eta \circ \mathbf{p}_k \circ \iota^{\otimes k}$ respectively.

The first important property to notice is that $\mathbf{p}_k$ is (by induction) superfiltered of degree $1$, so $\mu^k_\mathrm{min}$ is too.  Hence the $\mu^k_\mathrm{min}$ do indeed define a superfiltered $A_\infty$-structure on $E$, and it remains to check that it is strictly unital and that it reduces to the standard $A_\infty$-structure on the associated graded.  For the latter, take homogeneous elements $a_1, \dots, a_k$ in $E$.  We want to show that the leading term of $\mu^k_\mathrm{min}(a_k, \dots, a_1)$ is $(-1)^{|a_1|}a_2 \wedge a_1$ if $k=2$ (the sign comes from translating to the $A_\infty$-world by \eqref{eqdgAinfinity}, as usual) and is $0$ otherwise.  Letting $\approx$ denote equality of leading terms, for $k=2$ we have
\[
\mu^2_\mathrm{min}(a_2, a_1) = \pi \circ \mu^2_{\calB_0}(\iota (a_2), \iota (a_1)) \approx \pi \circ \iota((-1)^{|a_1|}a_2 \wedge a_1) = (-1)^{|a_1|}a_2 \wedge a_1,
\]
which is what we want.  Here the $\approx$ uses that fact that although $\iota$ is not an algebra homomorphism with respect to wedge product on $E$, it \emph{is} to leading order (i.e.~$\gr \iota$ is an algebra homomorphism), whilst the final equality uses $\pi \circ \iota = \id_E$.  For $k > 2$ note that if we unwind the inductive definition of $\mu^k_\mathrm{min}$, or more easily if we look at the tree description of $\mathbf{p}_k$ in \cite[Section 4]{Markl}, then each summand contains (possibly nested inside other applications of $\mu^2_{\calB_0}$ and $\eta$) an expression of the form
\begin{equation}
\label{eqLeadingTerm}
\eta \circ \mu^2_{\calB_0} (\iota (a_{i+1}), \iota (a_i)).
\end{equation}
Again using the fact that $\iota$ is an algebra homomorphism to leading order, and the fact that $\eta$ is zero by definition on the image of $\iota$, we see that the leading term of \eqref{eqLeadingTerm} is zero.

Finally we deal with strict unitality.  Since $\iota(1_E) = 1_{\calB_0}$ it is clear that $1_E$ is a unit for $\mu^2_\mathrm{min}$.  We now just need to check that $\mu^k_\mathrm{min}(a_k, \dots, a_1)$ vanishes if $k>2$ and some $a_i$ is equal to $1_E$.  To do this, note (e.g.~by considering the tree description) that each term in the expansion of $\mu^k_\mathrm{min}(a_k, \dots, a_1)$ contains an expression of one of the following forms:
\begin{gather*}
\eta\circ \mu^2_{\calB_0}(\iota(a_{i+1}), \iota(a_i)) = \eta \circ \iota(a_{i+1}),
\\ \eta\circ \mu^2_{\calB_0}(\eta(\bullet), \iota(a_i)) = \eta \circ \eta(\bullet),
\\ \pi\circ \mu^2_{\calB_0}(\eta(\bullet), \iota(a_i)) = \pi\circ (\eta(\bullet)),
\end{gather*}
or the corresponding things with $\iota(a_i)$ appearing as the left-hand input of $\mu^2_{\calB_0}$.  These all vanish by the side conditions from \cref{lemeta}.
\end{proof}

\begin{lem}
\label{lemPiProperties}
The map $\Pi$ is superfiltered of degree $0$ (meaning each $\Pi^k : \calB_0[1]^{\otimes k} \to \calB_0^\mathrm{min}[1]$ is superfiltered of degree $0$), and if $k \geq 2$ then the component $\Pi^k$ vanishes on $(\im \iota)^{\otimes k}$.
\end{lem}

\begin{rmk}
The vanishing statement is not needed right now but will be used later.
\end{rmk}

\begin{proof}
The formula for $\Pi^k$ is $\pi \circ \mathbf{q}_k$, where $\mathbf{q}_k : \calB_0^{\otimes k} \to \calB_0$ is defined inductively by $\mathbf{q}_1 = \id_{\calB_0}$ and
\begin{multline}
\label{eqqk}
\mathbf{q}_k (a_k, \dots, a_1) =\sum_{r\geq 1} \sum_{\substack{s_1, \dots, s_r \geq 1 \\ s_{r+1} \geq 0, 2-r \\ s_1+\dots+s_{r+1} = k}} \pm \mathbf{p}^r_{r+s_{r+1}} (a_k, \dots, a_{k-s_{r+1}+1},
\\ \eta \circ \mathbf{q}_{s_r} (a_{k-s_{r+1}}, \dots), \iota\pi \circ \mathbf{q}_{s_{r-1}} (\dots), \dots, \iota\pi\circ \mathbf{q}_{s_1} (\dots, a_1)).
\end{multline}
We have translated Markl's expressions to our ordering convention (his inputs are read left-to-right whilst ours are right-to-left), and do not need the precise signs.  The $\mathbf{p}^i_j$ themselves have an inductive definition, and since $\calB_0$ has vanishing higher $A_\infty$-operations (it's a dg-algebra) this simplifies to
\begin{equation}
\label{eqpij}
\mathbf{p}^i_j (a_j, \dots, a_1) = \begin{cases} \pm \mu_{\calB_0}^2 (a_2, a_1) & \text{if } i=1 \text{ and } j=2 \\ \pm \mu_{\calB_0}^2 (a_j, \eta \circ \mathbf{p}_{j-1}(a_{j-1}, \dots, a_1)) & \text{if } i=2 \\ 0 & \text{otherwise.}\end{cases}
\end{equation}
The map $\Pi^k$ then inherits superfilteredness of the correct degree from $\iota$, $\pi$, $\eta$, $\mu^2_{\calB_0}$, and the $\mathbf{p}_j$.

To prove the required vanishing property of $\Pi^k$ we shall show by induction on $k$ that $\mathbf{q}_k$ vanishes on $(\im \iota)^{\otimes k}$ for $k \geq 2$.  To do this, focus on the term $\eta \circ \mathbf{q}_{s_r}$ in \eqref{eqqk}.  Since $s_r$ is less than $k$, this term vanishes by induction if $s_r \geq 2$.  We are left to deal with the case $s_r = 1$, where the term is
\[
\eta \circ \mathbf{q}_1(a_{k-s_r+1}) = \eta (a_{k-s_r+1}).
\]
This vanishes since $\eta \circ \iota = 0$, completing the inductive step and proving the lemma.
\end{proof}

\begin{lem}
\label{lemPiSU}
The map $\Pi$ is strictly unital.
\end{lem}
\begin{proof}
The unit $1_{\calB_0}$ in $\calB_0$ is $\id_{E_R}$, and $\Pi^1 = \pi$ sends this to $1_E$ in $\calB_0^\mathrm{min} = E$.  We claim that $\mathbf{q}_k(a_k, \dots, a_1)$ lies in the image of $\eta$ whenever $k \geq 2$ and some $a_i$ is equal to $1_{\calB_0}$.  The lemma then follows from the definition of $\Pi^k$ as $\pi \circ \mathbf{q}_k$, in conjunction with the side condition $\pi \circ \eta = 0$.

To prove the claim, we begin by expanding out \eqref{eqqk} using \eqref{eqpij}:
\begin{multline}
\label{eqqkExpanded}
\mathbf{q}_k(a_k, \dots, a_1) = \pm \mu^2_{\calB_0}(a_k, \eta \circ \mathbf{q}_{k-1}(a_{k-1}, \dots, a_1))
\\ + \sum_{j=1}^{k-1} \pm \mu^2_{\calB_0}(\eta \circ \mathbf{q}_{k-j}(a_k, \dots, a_{j+1}) , \iota \pi \circ \mathbf{q}_j(a_j, \dots, a_1))
\\ + \sum_{\substack{s_1, s_2 \geq 1 \\ s_1 + s_2 \leq k-1}} \pm \mu^2_{\calB_0} (a_k, \eta \circ \mathbf{p}_{k-1}(a_{k-1}, \dots, a_{s_1+s_2+1},
\\ \eta \circ \mathbf{q}_{s_2}(a_{s_1+s_2}, \dots, a_{s_1+1}), \iota\pi \circ \mathbf{q}_{s_1}(a_{s_1}, \dots, a_1))).
\end{multline}
We now argue by induction on $k$, considering the possible positions where the $1_{\calB_0}$ can occur in each term on the right-hand side of \eqref{eqqkExpanded}---say we have $a_i = 1_{\calB_0}$.  In the first term: if $i=k$ then use unitality of $\mu^2_{\calB_0}$; if $i < k$ and $k=2$ then use the fact that $\mathbf{q}_1 = \id_{\calB_0}$ and that $\eta(1_{\calB_0}) = \eta(\iota(1_E)) = 0$; otherwise use the inductive hypothesis and the side condition $\eta^2 = 0$.  In the second term: if $i \geq j+1$ then similar arguments apply; otherwise use $\iota\pi \circ \mathbf{q}_1 (1_{\calB_0}) = 1_{\calB_0}$ and unitality of $\mu^2_{\calB_0}$ if $j = 1$, and the inductive hypothesis plus $\pi \circ \eta = 0$ if $j>1$.  For the third term: if $i = k$ then use unitality of $\mu^2_{\calB_0}$; if $i = s_1+1$ and $s_2=1$ then use $\eta(1_{\calB_0}) = 0$; and if $i \leq s_1+s_2$ and $s_1, s_2 \geq 2$ then use the inductive hypothesis and side conditions.  This leaves the cases $s_1+s_2+1 \leq i \leq k-1$ and $i = s_1 = 1$, and these follow from the fact that $\eta \circ \mathbf{p}_{\geq 2}$ vanishes whenever some input is $1_{\calB_0}$ (proved by a similar argument to the proof of strict unitality of $\calB_0^\mathrm{min}$ in \cref{lemB0min}).
\end{proof}


\subsection{Computing the disc potential}
\label{sscComputingPotential}

The last thing we need to do whilst studying $\calB_0^\mathrm{min}$ is to calculate its disc potential.  This computation was essentially done by Dyckerhoff (without signs or full proof) in \cite[Section 5.6]{Dyckerhoff}, and by Sheridan (in characteristic $0$) in \cite[Proposition 7.1]{SheridanCY}.

There is one small preliminary step we need first.

\begin{lem}
\label{lemetaleading}
For $r \in \m$, the leading term of $\eta(r \id_{E_R})$ is of the form $\sum_i m_i(r) v_i^\vee$, where the $m_i(r)$ are elements of $R$ satisfying $\sum_i x_i m_i(r) = -r$.
\end{lem}

Using this $m_i$ notation we can, and will, write the $w_i$ as $-m_i(w)$ and the $\lambda_{ij}$ appearing in the definition of the $f_i$ as $m_j(w_i)$.

\begin{proof}
When we constructed $\eta$ in \cref{lemeta} we first built it inductively from the $\eta_{-1}$ of \cref{lemetam1}, and then passed to $\eta - \diff \eta^3 \diff$ to satisfy the side conditions.  Since $\diff(r \id_{E_R}) = 0$, the passage to $\eta-\diff \eta^3 \diff$ makes no difference, and the leading term of $\eta(r \id_{E_R})$ is simply $\eta_{-1}(r \id_{E_R})$.  We can compute the latter explicitly using the recipe from \cref{lemetam1}.

First we split $r \id_{E_R}$ into its trigraded pieces, namely $r_p \id_{E_R}$ in tridegree $(p, 0, 0)$, where $r_p$ is the piece of $r$ lying in $R_p$.  Recall that $R_p$ comprises the homogeneous polynomials of degree $p$ in the $x_i$.  This $r_p \id_{E_R}$ lives in $K^{p, 0}$, in cohomological degree $0$, so its image under $\eta_{-1}$ lives in $K^{p, 0}$ in cohomological degree $-1$, i.e.~in tridegree $(p-1, 0, 1)$.  Thus it is of the form $\sum_i m_i(r_p) v_i^\vee$, where each $m_i(r_p)$ is an element of $R_{p-1}$.  The total $\eta_{-1}(r \id_{E_R})$ is then defined to be
\[
\sum_p \sum_i m_i(r_p) v_i^\vee.
\]
Letting $m_i(r) = \sum_p m_i(r_p)$, it remains to show that $\sum_i x_im_i(r) = -r$.  But this follows immediately from the homotopy condition $\mu^1_{E_\mathrm{dg}} \circ \eta_{-1} (r \id_{E_R}) = - r \id_{E_R}$.
\end{proof}

We can now compute the disc potential.

\begin{thm}
\label{propDiscPotential}
The disc potential $\Po_0$ of $\calB_0^\mathrm{min}$ is $w$.
\end{thm}
\begin{proof}
Introduce formal variables $t_1, \dots, t_n$, and let $\vv_t$ denote $\sum_i t_i v_i$; we will compute $\Po_0$ in terms of these new variables, rather than the $x_i$, since the latter already denote the variables in the ring $R$.  Similarly, let $\widehat{\vv}_t$ denote $\sum_i t_i f_i$.  Extending all operations multilinearly in the $t_i$ we obtain
\[
\Po_0(t) = \sum_k \mu^k_\mathrm{min} (\vv_t, \dots, \vv_t) = \sum_k \pi \circ \mathbf{p}_k(\widehat{\vv}_t, \dots, \widehat{\vv}_t).
\]

The leading term of $\mathbf{p}_k(\widehat{\vv}_t, \dots, \widehat{\vv}_t)$ lies in degree $2$ with respect to the $\Z$-grading on $\calB_0$, whilst the next term---which we'll call the \emph{sub-leading term}---lies in degree $0$.  A straightforward induction using \cref{lemetaleading} shows that for all $k \geq 2$ the leading term vanishes, whilst the sub-leading term is given by
\[
(-1)^k \sum_{i_1, \dots, i_k} t_{i_1} \cdots t_{i_k} m_{i_k}(\cdots(m_{i_1}(w))) \id_{E_R}.
\]

By definition, the disc potential comes from the sub-leading part of $\sum_k \mu^k_\mathrm{min}(\vv_t, \dots, \vv_t)$, and all other parts vanish.  By the previous paragraph we therefore have
\[
\Po_0(t) = \sum_{k \geq 2} (-1)^k \sum_{i_1, \dots, i_k} t_{i_1} \cdots t_{i_k} \pi(m_{i_k}(\cdots(m_{i_1}(w))) \id_{E_R}).
\]
Writing $w^{(p)}$ for the part of $w$ lying in  $R_p$, we get
\[
\Po_0(t) = \sum_{k \geq 2} (-1)^k \sum_{i_1, \dots, i_k} t_{i_1} \cdots t_{i_k} m_{i_k}(\cdots(m_{i_1}(w^{(k)}))).
\]
This is because $m_{i_k}(\cdots(m_{i_1}(w^{(p)}))) \id_{E_R}$ vanishes for $p < k$ and is killed by $\pi$ for $p > k$.  Hence
\[
\Po_0(x) = \sum_{k \geq 2} (-1)^k \sum_{i_1, \dots, i_k} x_{i_1} \cdots x_{i_k} m_{i_k}(\cdots(m_{i_1}(w^{(k)}))) = \sum_k w^{(k)} = w.\qedhere
\]
\end{proof}

As an immediate consequence we deduce

\begin{cor}
The map \eqref{eqMainBijection} is surjective.\hfill$\qed$
\end{cor}


\section{The matrix factorisation $\mfE$}
\label{secmfE}

For \cref{Theorem1} it remains to prove injectivity of \eqref{eqMainBijection}, and for this recall the strategy outlined in \cref{sscOutline}.  Given a superfiltered $A_\infty$-deformation $\calA = (\calA, (\mu^k)_{k \geq 2})$ of $E$ with disc potential $\Po \in \m^2 \subset R$ we shall build a filtered matrix factorisation $\mfE$ of $\Po$, and an $A_\infty$-homomorphism $\F$ from $\calA$ to $\calB = \eend_{\mffilt}(\mfE)$, using the localised mirror construction.  Letting $\mfE_0$ denote the matrix factorisation of $w$ considered above, with $w$ set to $\Po$, we shall then construct an $A_\infty$-homomorphism $\Psi$ from $\calB$ to $\calB_0 = \eend_{\mffilt}(\mfE_0)$, and combine this with the projection $\Pi$ from $\calB_0$ to its minimal model $\calB_0^\mathrm{min}$ to obtain an $\infty$-equivalence
\[
\calA \xrightarrow{\ \Pi \circ \Psi \circ \F \ } \calB_0^\mathrm{min}.
\]
Finally we deal with $d$-equivalences for $d < \infty$ by reducing to $d=\infty$ using formal diffeomorphisms.


\subsection{The localised mirror functor}
\label{sscLocalisedMirror}

Fix then for the rest of the section such a superfiltered deformation $\calA$.  Instead of keeping track of the $A_\infty$-isomorphism $\gr \calA \to E$ that is part of the data of the deformation, we instead fix an identification of $\calA$ with $E$ which coincides with the given isomorphism at the graded level (as in \cref{exAinftyDef}).  The leading terms of the operations $\mu^k$ on $\calA$ then coincide with the standard operations on $E$.  Let $\vv \in \m \otimes V$ be as in \cref{defDiscPot}, so its disc potential $\Po$ is given by
\[
\Po = \sum_k \mu^k(\vv, \dots, \vv),
\]
after extending the $\mu^k$-operations $R$-multilinearly.

In this subsection we construct, following Cho--Hong--Lau \cite[Section 2.3]{ChoHongLauLMF}, the filtered matrix factorisation $\mfE \in \mffilt(R, \Po)$ and a strictly unital superfiltered $A_\infty$-homomorphism
\[
\F : \calA \to \calB \coloneqq \eend_{\mffilt}(\mfE)
\]
over $R_0$, of degree $0$.  Note we use different $A_\infty$-conventions from \cite{ChoHongLauLMF} so our formulae have different signs.  Before getting into the construction we introduce modified degree $1$ operations
\[
\muzvv^k : \calA[1]^{\otimes k} \to R \otimes \calA[1],
\]
defined by extending the $\mu^k$ $R$-multilinearly again and setting
\[
\muzvv^k (a_k, \dots, a_1) = \sum_{l\geq 0} \mu^{k+l}(a_k, \dots, a_1, \vv, \dots, \vv).
\]
(The subscript $\mathbf{0},\vv$ indicates that we are summing over insertions of $\vv$ after the rightmost input, $a_1$, and over insertions of $0$ before and between all other inputs.  Of course, summing over insertions of $0$ does nothing, but later we will need to consider the case where we sum over insertions of $\vv$ before, after, and between all inputs, for which we will simply use a subscript $\vv$.)  By applying the $A_\infty$-relations to $a_k, \dots, a_1, \vv, \dots, \vv$, and using strict unitality, we obtain
\begin{multline}
\label{eqmvvRelations}
\sum_{i \geq 1,j} (-1)^{\maltese_i} \muzvv^{k-j+1}(a_k, \dots, a_{i+j+1}, \mu^j(a_{i+j}, \dots, a_{i+1}), a_i, \dots, a_1) \\ + \sum_{j\geq 1} \muzvv^{k-j+1} (a_k, \dots, a_{j+1}, \muzvv^j(a_j, \dots, a_1)) = \begin{cases} -\mu^2(a_1, \Po) = -\Po a_1 & \text{if } k=1 \\ 0 & \text{if } k > 1,\end{cases}
\end{multline}
where $\maltese_i$ denotes $|a_1| + \dots + |a_i| - i$ as usual.

With this in hand, we define $\mfE$ to have underlying superfiltered module $E_R = R \otimes E$ (as in \cref{sscDefE0}), and squifferential $\diff_\mfE$ given by $\diff_{\mfE}a = (-1)^{|a|} \muzvv^1(a)$ for all $a$.

\begin{lem}[{\cite[Theorem 2.19]{ChoHongLauLMF}}]
This $\mfE$ is indeed an object of $\mffilt(R, \Po)$.
\end{lem}
\begin{proof}
The map $\diff_\mfE$ is superfiltered of degree $1$ because the operations $\mu^k$ on $\calA$ are superfiltered of degree $1$.  It remains to check that $\diff_\mfE^2 = \Po$, which amounts to $-\muzvv^1(\muzvv^1(a)) = \Po a$ for all $a$, and this is just the $k=1$ case of \eqref{eqmvvRelations}.
\end{proof}

\begin{rmk}
The special case where $\calA$ is a Clifford algebra, with vanishing higher operations, may be familiar to the reader: in this situation the lemma simply says that wedging with (minus) the Euler vector field gives a matrix factorisation of the defining quadratic form.
\end{rmk}

The next task is to define the $A_\infty$-algebra map $\F : \calA \to \calB = \eend_{\mffilt}(\mfE)$.  For each $k \geq 1$ we thus need to give a degree $0$ map $\F^k : \calA[1]^{\otimes k} \to \calB[1]$.  Recalling that $\calA$ is identified with $E$, we have the following identifications of $R_0$-modules
\[
\calB = \End_R(E_R) = \Hom_{R_0}(E, E_R) = \Hom_{R_0}(\calA, R \otimes \calA) = \Hom_{R_0}(\calA[1], R \otimes \calA[1]).
\]
We may thus express $\F^k$ as a degree $1$ map
\[
\calA[1]^{\otimes k} \otimes \calA[1] \to R \otimes \calA[1].
\]
In this laguage, and still following \cite[Section 2.3]{ChoHongLauLMF} (modulo the change of conventions), we define
\begin{equation}
\label{eqLocalisedMirrorMap}
\F^k(a_k, \dots, a_1)(a_0) = (-1)^{|a_0|} \muzvv^{k+1}(a_k, \dots, a_1, a_0).
\end{equation}

\begin{lem}[{\cite[Theorem 2.19]{ChoHongLauLMF}}]
\label{lemPhiProperties}
This $\F$ is indeed an $A_\infty$-algebra map (over $R_0$), and is strictly unital and superfiltered of degree $0$.
\end{lem}
\begin{proof}
The fact that it is strictly unital and superfiltered of degree $0$ follow from strict unitality and superfilteredness of the $\mu^k$.  It remains to check the $A_\infty$-homomorphism relations, namely that for all tuples $a_1, \dots, a_k$ in $\calA$ (with $k \geq 1$) we have
\begin{multline}
\label{eqAinfinityHomCheck}
\sum_{i,j} (-1)^{\maltese_i} \F^{k-j+1}(a_k, \dots, a_{i+j+1}, \mu^j(a_{i+j}, \dots, a_{i+1}), a_i, \dots, a_1)
\\ = \sum_r \sum_{\substack{s_1, \dots, s_r \\ s_1+\dots+s_r = k}} \mu_\calB^r(\F^{s_r}(a_k, \dots, a_{k-s_r+1}), \dots, \F^{s_1}(a_{s_1}, \dots, a_1)).
\end{multline}
Applying the left-hand side to $a_0 \in \calA$ gives
\[
\sum_{i\geq 0,j} (-1)^{\maltese_i+|a_0|} \muzvv^{k-j+2}(a_k, \dots, a_{i+j+1}, \mu^j(a_{i+j}, \dots, a_{i+1}), a_i, \dots, a_0),
\]
which by \eqref{eqmvvRelations} we can rewrite as
\begin{equation}
\label{eqRelationToCheckLHS}
\sum_{j\geq 0} \muzvv^{k-j+1} (a_k, \dots, a_{j+1}, \muzvv^{j+1}(a_j, \dots, a_0)).
\end{equation}
The right-hand side of \eqref{eqAinfinityHomCheck}, meanwhile, only has terms with $r=1$ or $2$ (since $\calB$ has vanishing higher operations), so applying it to $a_0$ we get
\begin{multline*}
(-1)^{|a_0|}\muzvv^1 (\F^k(a_k, \dots, a_1)(a_0)) - (-1)^{|a_0|}\F^k(a_k, \dots, a_1)(\muzvv^1 (a_0)) \\+ \sum_{j=1}^{k-1} (-1)^{\maltese_j +1} \F^{k-j}(a_k, \dots, a_{j+1}) \circ \F^j(a_j, \dots, a_1) (a_0).
\end{multline*}
Plugging in the definitions of $\diff_\mfE$ and $\F$, this becomes
\[
\muzvv^1(\muzvv^{k+1}(a_k, \dots, a_0)) + \muzvv^{k+1}(a_k, \dots, a_1, \muzvv^1(a_0))+ \sum_{j=1}^{k-1} \muzvv^{k-j+1}(a_k, \dots, a_{j+1}, \muzvv^{j+1}(a_j, \dots, a_0)),
\]
which is exactly \eqref{eqRelationToCheckLHS}.
\end{proof}


\subsection{Comparing $\mfE_0$ and $\mfE$}
\label{sscmfComparison}

The above construction of $\mfE$, and the construction of $\mfE_0$ from \cref{sscDefE0} with $w = \Po$, provides us with two objects in the category $\mffilt(R, \Po)$, both with underlying module $E_R$.  The squifferential on $\mfE_0$ is defined explicitly by $\diff_{\mfE_0}a = -(\vv \wedge a + \ww \contract a)$, whilst that on $\mfE$ depends on the $A_\infty$-operations on our given deformed algebra $\calA$.  Note that although $\diff_\mfE$ is complicated, its leading term is the leading term of
\[
\sum_k(-1)^{|\bullet|} \mu^{k+1}(\bullet, \vv, \dots, \vv),
\]
which is just $-\vv \wedge \bullet$.  This is because the leading terms of $\mu^3, \mu^4, \dots$ vanish and that of $\mu^2$ coincides with $\wedge$ (up to a sign twist), by our definition of a superfiltered $A_\infty$-deformation.  We deduce

\begin{lem}
\label{lemDifferentialLeadingTerms}
The leading terms of $\diff_\mfE$ and $\diff_{\mfE_0}$ are both $a \mapsto - \vv \wedge a$.\hfill$\qed$
\end{lem}

The endomorphism dg-algebras $\calB$ and $\calB_0$ of $\mfE$ and $\mfE_0$ both have the same underlying $R$-algebra, namely $\End_R(E_R)$; it is only the differentials which are different, and we have just seen that even these agree to leading order.  Inspired by \cite[Theorem 9.1]{ChoHongLauTorus}, our goal in this subsection is to prove

\begin{prop}
\label{propmfEqual}
The obvious `identity' map $\calB \to \calB_0$ can be corrected (by adding lower order terms) to a unital dg-algebra isomorphism $\psi$, which is superfiltered of degree $0$.  This $\psi$ can thus be viewed as the first term in a strictly unital superfiltered $A_\infty$-homomorphism $\Psi$ over $R$, with $\Psi^{>1} = 0$.
\end{prop}

The key ingredient is

\begin{lem}
\label{lemidERCocycle}
The map $\id_{E_R}$ can be corrected to a cocycle $i$ in $\hom^0_{\mffilt}(\mfE, \mfE_0)$
\end{lem}

\begin{proof}[Proof of \cref{propmfEqual}]
Assuming \cref{lemidERCocycle}, one can inductively write down the two-sided inverse to $i$, which is automatically a cocycle of the form $\id_{E_R}+\text{(lower order terms)}$.  The required map $\psi$ is then $a \mapsto iai^{-1}$, where the multiplication takes place in the common underlying algebra.
\end{proof}

It remains to prove \cref{lemidERCocycle}, which we will do via a spectral sequence (cf.~\cite[Theorem 9.1]{ChoHongLauTorus}).

\begin{proof}[Proof of \cref{lemidERCocycle}]
Consider the superfiltered chain complex $C \coloneqq \hom_{\mffilt} (\mfE, \mfE_0)$, and the $\Z$-graded filtered complex $C^T$ built from it as in \cref{sscSetup}.  The first page of the associated spectral sequence is $R_0[T^{\pm 1}] \otimes \rH^*(C, \diff_1)$, where $\diff_1$ is the leading term of $\diff_C$.  Since $\diff_\mfE$ and $\diff_{\mfE_0}$ both have leading term $-\vv \wedge \bullet$, the complex $(C, \diff_1)$ coincides with $E_\mathrm{dg}$ studied in \cref{sscEdg}, so we have $\rH^*(C, \diff_1) \cong \rH^*(E_\mathrm{dg}) \cong E$ as $R_0$-modules.

The map $\id_{E_R}$ in $C$ is a $\diff_1$-cocycle and corresponds to the cohomology class $1 \in E$.  We can view this element as lying in the zeroth column and zeroth row on the first page of the spectral sequence, and the claimed result is equivalent to its differential on each page being zero.  To see that this is indeed the case, note that on the $r$th page its differential lands in the $r$th column, $(1-r)$th row, and the group in this position is zero, even on the first page.
\end{proof}

\begin{rmk}
In principle the object $\mfE_0$ depends on the choice of $\ww$, but the argument used to prove \cref{propmfEqual} also shows that any two choices give rise to isomorphic objects in the category $Z^0\mffilt(R, w)$, whose morphisms are cocycles in $\mffilt(R, w)$ which are superfiltered of degree $0$.  Moreover, the isomorphism can be chosen to be $\id_{\End_R(E_R)}$ to leading order.
\end{rmk}

\begin{rmk}
One can view $R$ as a curved dg- (cdg-)algebra, with differential zero and curvature $w$.  The equivalence between $\mfE$ and $\mfE_0$ should then be a consequence of the fact that both are projective resolutions of $R/\m$ as a cdg-$R$-module (I thank an anonymous referee for this comment).  We prefer to give the direct argument above because of its explicitness.  For example, the cocycle $i$ in \cref{lemidERCocycle} appears in \cref{secMonotoneTori} as a `quantum change of variables'.
\end{rmk}


\subsection{Completing the proof of \cref{Theorem1}}

Putting everything together, we have superfiltered $A_\infty$-deformations $\calA$ and $\calB_0^\mathrm{min}$ of the exterior algebra $E$, and $A_\infty$-algebra maps over $R_0$
\[
\calA \xrightarrow{\ \F\ } \calB \xrightarrow{\ \Psi\ } \calB_0 \xrightarrow{\ \Pi\ } \calB_0^\mathrm{min}.
\]
Here $\calB$ is the endomorphism algebra of the matrix factorisation $\mfE$ and $\F$ is the map from the localised mirror functor; $\calB_0$ is the endomorphism algebra of the reference matrix factorisation $\mfE_0$ and $\Psi$ is the comparison isomorphism from \cref{propmfEqual}; and $\Pi$ is the projection from $\calB_0$ to its minimal model $\calB_0^\mathrm{min}$.  Our first goal is to show that $\calA$ is $\infty$-equivalent to $\calB_0^\mathrm{min}$, which depends on $\calA$ only through its potential.

\begin{prop}
\label{propCompositionEquivalence}
The map $\Pi \circ \Psi \circ \F: \calA \to \calB_0^\mathrm{min}$ is an $\infty$-equivalence.
\end{prop}
\begin{proof}
Recall that $\calB_0^\mathrm{min}$ has underlying module $E$, and we have also identified $\calA$ with $E$.  So, taking the perspective of \cref{exSuperfilteredEquivalence}, an $\infty$-equivalence $\chi$ is a strictly unital superfiltered $A_\infty$-homomorphism such that the leading term of $\chi^r$ is $\id_E$ for $r=1$ and vanishes for $r \geq 2$.  The map $\chi \coloneqq \Pi \circ \Psi \circ \F$ is automatically a strictly unital superfiltered $A_\infty$-homomorphism, since $\F$, $\Psi$ and $\Pi$ are so (see \cref{lemPhiProperties,propmfEqual,lemPiProperties,lemPiSU}).  It therefore remains to compute the leading term.

From \eqref{eqLocalisedMirrorMap} we have $\F^1(a_1)(a_0) = (-1)^{|a_0|}\muzvv^2(a_1, a_0)$ for all $a_0$ and $a_1$.  By an analogous argument to that used for \cref{lemDifferentialLeadingTerms}, we thus have
\[
\F^1(a_1)(a_0) \approx (-1)^{|a_0|} \mu^2(a_1, a_0) \approx a_1 \wedge a_0,
\]
where, as before, $\approx$ denotes equality of leading terms.  Similarly, for $k>1$ we have
\[
\F^k(a_k, \dots, a_1)(a_0) = (-1)^{|a_0|} \muzvv^{k+1}(a_k, \dots, a_0) \approx 0.
\]

Next, recall from \cref{sscmfComparison} that $\Psi^1$ is given by $\psi$, which is conjugation by the element $i$ in the algebra $\End_R(E_R)$ that underlies both $\calB$ and $\calB_0$.  This element is $\id_{E_R}$ to leading order, so combining this with the previous paragraph we obtain
\begin{equation}
\label{eqPsiPhi1}
(\Psi \circ \F)^1 (a) \approx a \wedge \bullet
\end{equation}
for all $a$ in $\calA = E$.  Meanwhile, $\Psi^k$ is defined to be zero for $k>1$, so (again using the previous paragraph) we have $(\Psi \circ \F)^k \approx 0$ for such $k$.

Turning now to $\Pi$, recall from \cref{sscEndmfE0} that $\Pi^1$ is given by $\pi$, which sends $f \in \End_R(E_R)$ to the reduction of $f(1)$ modulo $\m$.  Plugging this into \eqref{eqPsiPhi1} we obtain
\[
(\Pi \circ \Psi \circ \F)^1 (a) \approx a
\]
for all $a$.  Finally consider the leading term of $(\Pi \circ \Psi \circ \F)^k(a_k, \dots, a_1)$ for $k>1$.  Since the leading terms of the $(\Psi \circ \F)^{>1}$ vanish, the only contribution we need consider is that from
\[
\Pi^k \big((\Psi \circ \F)^1(a_k), \dots, (\Psi \circ \F)^1(a_1) \big) \approx \Pi^k(a_k \wedge \bullet, \dots, a_1 \wedge \bullet) \approx \Pi^k(\iota(a_k), \dots, \iota(a_1)),
\]
where $\iota$ is the map $E \to \End_R(E_R)$ from \cref{sscEndmfE0}.  By \cref{lemPiProperties} the right-hand side is zero.
\end{proof}

We deduce

\begin{cor}
\label{corMainThmInfty}
Any two superfiltered $A_\infty$-deformations of $E$ whose disc potentials are $\infty$-equivalent are themselves $\infty$-equivalent.  Combined with \cref{corSurjective}, this proves \cref{Theorem1} for $d=\infty$.
\end{cor}
\begin{proof}
\Cref{propCompositionEquivalence} shows that both algebras are $\infty$-equivalent to the same $\calB_0^\mathrm{min}$.
\end{proof}

It is now a simple matter to complete the proof of \cref{Theorem1}.

\begin{prop}
\label{propMainThm}
\Cref{corMainThmInfty} holds with $\infty$ replaced by any $d$ in $\{0, 1, 2, \dots, \infty\}$.
\end{prop}
\begin{proof}
Suppose $\calA_1$ and $\calA_2$ are superfiltered deformations and $f : V \to V$ is a $d$-equivalence between their potentials, so $\Po_1 = \Po_2 \circ f$.  The linear part of $f$ defines a linear autmorphism of $V$, and thus induces a linear automorphism of $E$ which we denote by $\Delta^1$.  For each $k \geq 2$ the homogeneous degree $k$ part of $f$ gives a map $V^{\otimes k} \to V$, and we extend this to a degree $0$ map $\Delta^k : E[1]^{\otimes k} \to E[1]$ using the projection-onto and inclusion-of the degree $1$ part of $E$ (with its $\Z$-grading).  We can view this $\Delta$ as a formal diffeomorphism of $\calA_1$ (`$\Delta$' stands for `diffeomorphism'), meaning a sequence of maps $\calA_1[1]^{\otimes k} \to \calA_1[1]$ whose $k=1$ component is a linear automorphism.  There is then a unique $A_\infty$-algebra structure on $E$, denoted by $\Delta_* \calA_1$, such that $\Delta$ defines an $A_\infty$-isomorphism $\calA_1 \to \Delta_* \calA_1$ \cite[Section (1c)]{SeidelBook}.

Since $\Delta$ is strictly unital and superfiltered of degree $0$, we have that $\Delta_* \calA_1$ is a superfiltered $A_\infty$-deformation of $E$.  Moreover, by construction $\Delta$ gives a $d$-equivalence $\calA_1 \to \Delta_* \calA_1$, and the potential $\Po_{1, \Delta}$ of $\Delta_* \calA_1$ satisfies
\[
\Po_1 = \Po_{1, \Delta} \circ f_{\Delta} = \Po_{1, \Delta} \circ f,
\]
where $f_\Delta$ is the change of variables as defined after \cref{rmkBeforeChangeOfVariables}.  Plugging in $\Po_1 = \Po_2 \circ f$, we see that $\Po_2 = \Po_{1,\Delta}$ and hence that $\calA_2$ is $\infty$-equivalent to $\Delta_* \calA_1$ by \cref{corMainThmInfty}.  Composing the $d$-equivalence $\Delta : \calA_1 \to \Delta_* \calA_1$ with an $\infty$-equivalence $\Delta_* \calA_1 \to \calA_2$, we obtain the desired $d$-equivalence $\calA_1 \to \calA_2$.
\end{proof}


\section{Alternative approaches and related results}
\label{secObstructionDeformation}

In this section we digress briefly to discuss a different method one could use to classify superfiltered $A_\infty$-deformations of $E$.  We then describe some similar results obtained without the superfiltered hypothesis.  These demonstrate how useful superfilteredness is in simplifying the problem.


\subsection{An alternative classification strategy: obstruction theory}
\label{sscObstructionTheory}

Given two superfiltered deformations $\calA_1$ and $\calA_2$ of $E$ with the same (i.e.~$\infty$-equivalent) disc potentials $\Po_1 = \Po_2$, our argument shows that they are $\infty$-equivalent by proving that each of them is $\infty$-equivalent to the matrix factorisation algebra $\calB_0^\mathrm{min}$.  Another possible approach is to start with the identity isomorphism $\id_E$ between the Clifford algebras $\rH^*(\calA_1)$ and $\rH^*(\calA_2)$---this holds because the quadratic parts of $\Po_1$ and $\Po_2$ coincide---and try to lift this step-by-step to an $\infty$-equivalence $\F = (\F^k)$.  A similar problem was considered by Seidel in \cite[Section (1g)]{SeidelBook}, where he proved the following result

\begin{prop}[{Simplified version of \cite[Lemma 1.9]{SeidelBook}}]
If $\Z$-graded $A_\infty$-algebras $\calA_1$ and $\calA_2$ have the same cohomology algebra $A$, and if the Hochschild cohomology groups $\HH^2(A)^{2-r}$ vanish for all $r \geq 3$, then there exists an $A_\infty$-map $\F : \calA_1 \to \calA_2$ inducing the identity map on cohomology.
\end{prop}
\begin{proof}[Sketch proof]
Suppose we have constructed $\F^1, \dots, \F^{r-1}$ so that the $A_\infty$-homomorphism equations are satisfied up to and including order $r-1$, for some $r \geq 3$.  Take an arbitrary $\F^r$ and consider the difference between the two sides of the order-$r$ $A_\infty$-homomorphism equation
\begin{equation}
\label{eqAinftyHom}
\sum \mu_{\calA_2} (\F(\dots),\dots,\F(\dots)) - \sum \pm \F(\dots, \mu_{\calA_1}(\dots),\dots).
\end{equation}
This descends to a cocycle in $CC^2(A)^{2-r}$, defined using the bar complex, and by assumption there exists a Hochschild cochain $\sigma$ with $\diff\sigma = \text{\eqref{eqAinftyHom}}$.  The latter says precisely that if $\F^{r-1}$ is replaced with $\F^{r-1}-\sigma$ then \eqref{eqAinftyHom} vanishes at cohomology level (the lower order equations are not affected).  One can then modify $\F^r$ to make \eqref{eqAinftyHom} vanish at chain level, and proceed by induction.
\end{proof}

In our case the $A_\infty$-algebras $\calA_1$ and $\calA_2$ are only $\Z/2$-graded, and hence so is the Hochschild cohomology.  Moreover, the groups $\HH^2(A)^{2-r}$ are in general non-zero: assuming for now that the quadratic part of $\Po_1=\Po_2$ vanishes, so that $A$ is just $E = \Lambda V$, the HKR theorem tells us that
\begin{equation}
\label{eqExteriorHKR}
\HH^t(A)^s \cong R_{t-s} \otimes_{R_0} \Lambda^t V.
\end{equation}
Here $R_{t-s}$ denotes the homogeneous degree $t-s$ part of $R = R_0\llbracket x_1, \dots, x_n \rrbracket$ as usual, and the grading on $\Lambda V$ is understood modulo $2$.  (If one is worried about applying HKR over a ground ring that isn't a field, one can explicitly resolve $A$ as an $A$-bimodule using the Koszul resolution and see \eqref{eqExteriorHKR} directly.)  In terms of the bar complex, the HKR map is given \cite[Equation (3.13)]{SeidelGenusTwo} by
\begin{equation}
\label{eqHKRBar}
\phi \in CC^t(A) \mapsto \sum_{r\geq 0}\phi^r(\vv, \dots, \vv) \in R \otimes_{R_0} \Lambda^t V.
\end{equation}

Now suppose we have constructed the first $r-1$ terms $\F^1, \dots, \F^{r-1}$ of an $\infty$-equivalence $\calA_1 \to \calA_2$, so that the $A_\infty$-homomorphism equations $\calA_1 \to \calA_2$ are satisfied up to and including order $r-1$.  The difference \eqref{eqAinftyHom} again defines a Hochschild cocycle, and under \eqref{eqHKRBar} it is sent to the degree $r$ part of $\Po_2 \circ f_\F - \Po_1$.  By our inductive hypothesis, the leading term of $\F^1$ is $\id_E$ whilst the leading terms of higher $\F^k$ vanish, so $f_\F$ is the identity.  Our assumption that $\Po_1 = \Po_2$ then ensures that the obstruction class $\Po_2 \circ f_\F - \Po_1$ vanishes, so we can pick a $\sigma$ which cobounds and replace $\F^{r-1}$ with $\F^{r-1}-\sigma$ to make the order-$r$ $A_\infty$-homomorphism equation hold.  There is no need to modify $\F^r$---which doesn't even appear in the equation, and which we may as well take to be zero---since $\mu^1=0$.  Then continue inductively.

To make this precise one needs to check that $\sigma$ can be chosen to be strictly unital, and to respect the filtration and vanish at the associated graded level.  The former can be achieved by working with the \emph{reduced} bar complex, whilst for the latter one can introduce a formal variable $T$ of degree $2$ and insert appropriate powers of $T$ into all formulae to restore $\Z$-gradings (as in \cref{defTcomplex}).  The expression \eqref{eqAinftyHom} is then divisible by $T$, as well as being a cocycle, and one needs to show that it is of the form $T\diff\sigma$.  In the present case ($A=E$) this is clear since no $T$'s appear in the product on $A$ and hence the Hochschild complex splits as a direct sum over powers of $T$.  To extend the whole argument to the general case, where the quadratic part of $\Po_1 = \Po_2$ is non-zero and so $A$ is a non-trivial Clifford algebra, one would need to compute the Hochschild cohomology with $T$ adjoined and check that \eqref{eqAinftyHom} is of the required form.

We do not pursue this, since our localised mirror functor approach has the advantage of being more direct and geometric, realising our $A_\infty$-algebra concretely in the dg-category of matrix factorisations, and in doing so providing a simple dg-model for it.  It is also better suited to our geometric goals in \cref{secMonotoneTori}.  However, the two approaches are actually more closely related than they may appear: Seidel describes his argument as a `nonlinear analogue of a spectral sequence', and the localised mirror functor can be seen as linearising it to the (ordinary) spectral sequence we use to compare the dg-algebras $\calB$ and $\calB_0$.


\subsection{Related results: deformation theory}

$A_\infty$-deformations of the exterior algebra have been studied before in the context of mirror symmetry, from the perspective of formal deformation theory and without the strong hypothesis of superfilteredness (i.e.~without assuming that the $A_\infty$-operations respect the obvious filtration and reduce to the standard operations on $E$ at the associated graded level).  The approach pioneered by Seidel \cite[Sections 3--5]{SeidelGenusTwo} begins with the differential graded Lie algebra (dgla) of Hochschild cochains, which governs $A_\infty$-deformations, and applies Kontsevich's formality theorem \cite{KontsevichDeformationQuantisation} and HKR to replace it with the dgla of polyvector fields.  This is essentially the right-hand side of \eqref{eqExteriorHKR}, equipped with the Schouten bracket and a grading shift, but there is also an extra formal deformation parameter $\hbar$, with respect to which one takes a completion.  Any given deformation can then be described by a gauge-equivalence class of polyvector fields, and the goal is to identify this class from the computation of a finite number of $A_\infty$-operations.  That this is a reasonable task is a consequence of finite determinacy for singularities \cite{Tougeron}: a formal function with an isolated critical point can be identified up to formal change of variables by a finite number of terms in its expansion.

This technique is very powerful, underpinning (amongst other applications) proofs of homological mirror symmetry by Seidel \cite{SeidelGenusTwo}, Efimov \cite{EfimovHigherGenus} and Sheridan \cite[see in particular Theorem 2.91]{SheridanCY}.  However, it requires significant ingenuity, and manipulations that are often specific to the situation at hand, exploiting additional constraints on the algebra that are known to exist for geometric reasons.   It also relies heavily on working in characteristic zero, both for the framing of the deformation problem in terms of dgla's (or $L_\infty$-algebras) and for the formality theorem to hold.

To end this discussion we mention the following theorem of Efimov

\begin{thm}[{\cite[Theorem 8.1]{EfimovHigherGenus}}]
Suppose $R_0$ is a field of characteristic zero and $w \in R_0[V]$ is a polynomial with no terms of degree two or lower.  Consider the matrix factorisation $\mfE_0$ of $w$ from \cref{sscDefE0}, and the minimal model $\calB_0^\mathrm{min}$ of its endomorphism dg-algebra.  The equivalence class of $\calB_0^\mathrm{min}$, under $\Z/2$-graded $A_\infty$-quasi-isomorphisms which act trivially on cohomology, determines $w$ up to a formal change of variables whose linear term is the identity.
\end{thm}

The corresponding result in the superfiltered world is the easy observation that the map \eqref{eqMainBijection} is well-defined when $d=1$ (plus \cref{thmB0Potential}).  In contrast, Efimov's result is highly non-trivial.


\section{Hochschild cohomology}
\label{secHH}

We conclude the purely algebraic part of the paper by computing the Hochschild cohomology $\HH^*(\calA)$ of an arbitrary superfiltered $A_\infty$-deformation $\calA$ of $E$.  This doesn't rely on our earlier results, but is similar in spirit: we repeatedly insert the class $\vv$ to convert the Hochschild complex into a more manageable one (just as we converted $\calA$ into the endomorphism algebra of a matrix factorisation using the localised mirror functor), and then relate this simpler complex to a standard construction (as we related the target matrix factorisation to $\mfE_0$).


\subsection{The reduced Hochschild complex}

Recall that the Hochschild cohomology $\HH^*(\calA)$ of an $A_\infty$-algebra $\calA$ over a ring $R_0$ is the $\Ext$-algebra of the diagonal bimodule, i.e.~$\Ext_{\calA \otimes \calA^\mathrm{op}} (\calA, \calA)$.  As usual, undecorated tensor products are over $R_0$.  When $\calA$ is free as an $R_0$-module, this can be computed by using the bar resolution of the diagonal to give the following Hochschild complex:
\[
CC^t(\calA) = \prod_{r \geq 0} \Hom^t_{R_0} (\calA[1]^{\otimes r}, \calA) = \prod_{r \geq 0} \Hom^{t-r}_{R_0} (\calA^{\otimes r}, \calA).
\]
The differential of $\phi = (\phi^r)_{r \geq 0} \in CC^t(\calA)$ given by
\begin{multline*}
(\diff \phi) (a_k, \dots, a_1) = \sum_{i,j} (-1)^{(t-1)\maltese_i} \mu^{k-j+1}(a_k, \dots, a_{i+j+1}, \phi^j(a_{i+j}, \dots, a_{i+1}), a_i, \dots, a_1)
\\ - \sum_{i,j} (-1)^{\maltese_i+t-1} \phi^{k-j+1}(a_k, \dots, a_{i+j+1}, \mu^j(a_{i+j}, \dots, a_{i+1}), a_i, \dots, a_1).
\end{multline*}
Here $t$ is in $\Z$ or $\Z/2$, depending on whether $\calA$ is $\Z$- or $\Z/2$-graded.  If $\calA$ is strictly unital then we can equivalently use the \emph{reduced} Hochschild complex
\[
\overline{CC}^t(\calA) = \{\phi \in CC^t : \phi^k(a_k, \dots, a_1) = 0 \text{ if some } a_i \text{ is an } R_0\text{-scalar multiple of the unit}\},
\]
with the restriction of the above differential.  We shall focus on this case from now on.


The cup product on $\overline{CC}(\calA)$ is defined by
\begin{multline}
\label{eqHHcup}
(\phi \smile \psi)^k(a_k, \dots, a_1) = \sum (-1)^{(\lvert \phi \rvert -1) \maltese_l + (\lvert \psi \rvert-1) \maltese_i}\mu^{k-j-m+1}(a_k, \dots, a_{l+m+1},
\\ \phi^m(a_{l+m}, \dots, a_{l+1}), a_l, \dots, a_{i+j+1}, \psi^j(a_{i+j}, \dots, a_{i+1}), a_i, \dots, a_1).
\end{multline}
This makes $\HH^*(\calA)$ into an associative unital algebra, and Gerstenhaber \cite[Corollary 1]{Gerstenhaber} showed that it is in fact graded-commutative; for an alternative explanation see \cite[Section 2.5]{SuarezAlvarez}.  One can make $\overline{CC}(\calA)$ into an $A_\infty$-algebra by defining higher $A_\infty$-operations with formulae analogous to \eqref{eqHHcup}.

\begin{rmk}
For a detailed exposition of Hochschild invariants of associative algebras and $A_\infty$-categories we refer the reader to Witherspoon \cite{WitherspoonIntroductionToHH} and Ritter--Smith \cite[Section 2]{RitterSmith} respectively.
\end{rmk}

Now assume that $\calA$ is a superfiltered $A_\infty$-deformation of $E = \Lambda V$.  As in \cref{sscLocalisedMirror}, we fix an identification $\calA \cong E$, lifting the given isomorphism $\gr \calA \cong E$.  To compute $\HH^*(\calA)$ from $\overline{CC}(\calA)$ we will write down a map to a smaller complex and show that it's a quasi-isomorphism by what is effectively another spectral sequence argument.  To define the necessary filtration, consider the following bigrading on $\overline{CC}(\calA)$:
\[
\overline{CC}^{r,s}(\calA) = \Hom^s_{R_0} (\calA^{\otimes r}, \calA) \quad \text{so that} \quad \overline{CC}(\calA) = \prod_{r \geq 0} \ \bigoplus_{s=-rn}^n \overline{CC}^{r,s}(\calA).
\]
Recall that $n$ is the rank of the module $V$, whose exterior algebra we're deforming.  The overall grading $t$ is $r+s$.  The components of the differential have bidegree $(i, 1-i-2j)$ for $i \geq 1$, $j \geq 0$ (if $\calA$ were $\Z$-graded then we could restrict to $j=0$), so in particular the differential preserves the decreasing filtration
\[
\overline{CC} = F^0\overline{CC} \supset F^1\overline{CC} \supset F^2\overline{CC} \supset \dots
\]
given by bounding the second grading, namely
\[
F^p \overline{CC} = \prod_{r \geq 0} \ \bigoplus_{s=-rn}^{n-p} \overline{CC}^{r,s}.
\]
We'll call this the \emph{main} filtration.

\begin{rmk}
If $\calA$ is a $\Z$-graded associative algebra $A$ then the full bigrading descends to cohomology $\HH^*(A)$, and the group $\HH^t(A)^s$ appearing in \cref{sscObstructionTheory} is the $(t-s, s)$ bigraded piece.
\end{rmk}

\begin{rmk}
One could also define a decreasing filtration
\[
\overline{CC} = L^0\overline{CC} \supset L^1\overline{CC} \supset \dots
\]
via the $r$-grading, namely
\[
L^p \overline{CC} = \{\phi \in \overline{CC} : \phi^r = 0 \text{ for } r \leq p\},
\]
and again this is respected by the differential.  This is called the \emph{length} filtration, and it will also make an appearance in our story.
\end{rmk}


\subsection{Constructing the main map}

The Hochschild--Kostant--Rosenberg theorem tells us that $\HH^*(E) \cong E_R = R \otimes E$ as algebras.  Moreover, the isomorphism is induced by the chain map
\[
\phi \in \overline{CC}(E) \mapsto \sum_{k \geq 0} \phi^k(\vv, \dots, \vv),
\]
after extending $\phi$ $R$-multilinearly, and this is in fact an $A_\infty$-homomorphism.  This map can be understood in more general terms, and by applying the same construction to our superfiltered deformation $\calA$ we will arrive at the desired map out of $\overline{CC}(\calA)$.  We now explain this generalisation.

The first ingredient is the observation that the length filtration already gives us an $A_\infty$-algebra map from $\overline{CC}(\calA)$ to a much simpler complex, namely $\calA$ itself, by `projecting to length zero', i.e.~projecting from $\overline{CC}(\calA)$ to $\overline{CC}(\calA) / L^1\overline{CC}(\calA)$.  This induces a map $\HH^*(\calA) \to \rH^*(\calA)$, but usually this will be far from injective since we've ignored most of $\overline{CC}$.  It will often not be surjective either, since $\HH^*(\calA)$ is graded-commutative but the Clifford algebra $\rH^*(\calA)$ is not in general.

The second ingredient is that by `inserting $\vv$ in all possible ways' one can define $A_\infty$-operations $\muvv^k : (R \otimes \calA[1])^{\otimes_R k} \to R \otimes \calA[1]$ on $R \otimes \calA$ (which we've identified with $E_R$) by
\[
\muvv^k(a_k, \dots, a_1) = \sum_{i_0, \dots, i_k} \mu^{k+i_0+\dots+i_k}(\underbrace{\vv, \dots, \vv}_{i_k}, a_k, \underbrace{\vv, \dots, \vv}_{i_{k-1}}, \dots, a_1, \underbrace{\vv, \dots, \vv}_{i_0});
\]
we denote this $A_\infty$-algebra by $\calA_{\vv}$.  It should be viewed as a deformation of $\calA$ but since we are already viewing $\calA$ itself as a deformation of $E$ we will avoid this terminology.  Crucially, there is a map from $\overline{CC}_{\! R_0}(\calA)$ to $\overline{CC}_{\! R}(\calA_{\vv})$, where the subscripts on $\overline{CC}$ denote the base ring (which will be taken as read from now on).  This map is again given by inserting $\vv$ in all possible ways, after extending Hochschild cochains $R$-multilinearly.  Explicitly, given a class $\phi$ in $\overline{CC}(\calA)$, its image $\phi_{\vv}$ under this map satisfies
\begin{equation}
\label{eqHHwbc}
\phi_{\vv}^k (a_k, \dots, a_1) = \sum_{i_0, \dots, i_k} \phi^{k+i_0+\dots+i_k}(\underbrace{\vv, \dots, \vv}_{i_k}, a_k, \underbrace{\vv, \dots, \vv}_{i_{k-1}}, \dots, a_1, \underbrace{\vv, \dots, \vv}_{i_0})
\end{equation}
for all $a_1, \dots, a_k$.  This defines an $A_\infty$-algebra map $\overline{CC}(\calA)[1] \to \overline{CC}(\calA_{\vv})[1]$, whose higher components $\overline{CC}(\calA)[1]^{\otimes \geq 2} \to \overline{CC}(\calA_{\vv})[1]$ are zero.

If we apply this map before projecting to length zero, i.e.~consider the composition
\begin{equation}
\label{eqHHmap}
P_{\vv} : \overline{CC}(\calA) \xrightarrow{\ \text{insert }\vv \text{, i.e.~$\phi \mapsto \phi_{\vv}$}\ } \overline{CC}(\calA_{\vv}) \xrightarrow{\ \text{project to length zero}\ } \calA_{\vv},
\end{equation}
we obtain a strictly unital $A_\infty$-algebra map $\overline{CC}(\calA) \to \calA_{\vv}$, and hence a unital algebra homomorphism $\HH^*(\calA) \to \rH^*(\calA_{\vv})$.  This has more chance of being injective, and our main result is that in fact

\begin{thm}
\label{thmHHisom}
The map $\rH^*(P_{\vv}) : \HH^*(\calA) \to \rH^*(\calA_{\vv})$ induced by $P_{\vv}$, which is given explicitly by
\[
[\phi] \mapsto \Big[\sum_{k\geq 0} \phi^k(\vv, \dots, \vv)\Big],
\]
is a canonical isomorphism of unital $R_0$-algebras.
\end{thm}

\begin{rmk}
The map \eqref{eqHHwbc} is a special case of the map on Hochschild cohomology given by incorporating bounding cochains; see Sheridan \cite[Section 4.2]{SheridanFano}.  The idea of pulling information down the length filtration in this way was inspired by Seidel---\cite[Equation (1.16)]{SeidelFlux} is precisely \eqref{eqHHmap}.
\end{rmk}


\subsection{The map is an isomorphism}

We thus have the reduced Hochschild complex $\overline{CC}(\calA)$, with its main filtration, and the above chain map $P_{\vv}$ to $\calA_{\vv}$.  In order to prove \cref{thmHHisom} we shall define a corresponding filtration on $\calA_{\vv}$ so that $P_{\vv}$ becomes a filtered chain map, then show that it is a quasi-isomorphism on the associated graded complexes, and finally use completeness of the filtrations to deduce that it's a quasi-isomorphism between the original complexes.

First we introduce a bigrading on $\calA_{\vv}$, whose underlying module is identified with $E_R = R \otimes \Lambda V$, by setting $\calA_{\vv}^{r,s} = R_r \otimes \Lambda^{r+s} V$, where $R_r$ comprises homogeneous polynomials of degree $r$ in the $x_i$.  Just as for $\overline{CC}^{r,s}$, the overall grading is $r+s$ and the components of the differential have bidegree $(i, 1-i-2j)$ for $i\geq 1$, $j\geq 0$.  Moreover, $P_{\vv}$ sends $\overline{CC}^{r,s}(\calA)$ to $\calA_{\vv}^{r,s}$.  We deduce that if $\calA_{\vv}$ is equipped with the decreasing filtration
\[
F^p\calA_{\vv} = \prod_{r \geq 0} \bigoplus_{s=-r}^{n-p} \calA_{\vv}^{r,s}
\]
then it becomes a filtered complex, and $P_{\vv}$ a filtered map.

\begin{lem}
\label{lemgrPqis}
The map $P_{\vv}$ induces an isomorphism of $R_0$-algebras
\[
\rH^*(\gr P_{\vv}) : \rH^*(\gr \overline{CC}(\calA)) \to \rH^*(\gr \calA_{\vv}).
\]
\end{lem}
\begin{proof}
For each $p$, the $p$th associated graded piece $\gr^p \overline{CC} = F^p\overline{CC} / F^{p+1}\overline{CC}$ is
\[
\prod_{r\geq 0} \Hom^{n-p}_{R_0} (\calA^{\otimes r}, \calA).
\]
The induced differential is the bidegree $(1,0)$ component of the original differential, and this is precisely the Hochschild differential for the ordinary exterior algebra $E$.  Similarly, we have
\[
\gr^p \calA_{\vv} = \prod_{r \geq 0} R_r \otimes \Lambda^{n-p+r} V,
\]
with differential given by the leading term of $\mu^2(\vv, \bullet) + \mu^2(\bullet, \vv)$, which is zero.  The fact that $\rH^*(\gr P_{\vv})$ is an isomorphism then follows from the classical HKR theorem for the exterior algebra, since the latter respects the $(r,s)$ bigrading.
\end{proof}

Now we can deduce what we actually want, namely that $\rH^*(P_{\vv})$ itself is an isomorphism:

\begin{proof}[Proof of \cref{thmHHisom}]
The main filtration on $\overline{CC}$ is $I$-complete and $P$-complete in each degree, meaning respectively that for each $t$ the natural maps of $R_0$-modules
\[
\varinjlim_p F^p\overline{CC}^t \to \overline{CC}^t \quad \text{and} \quad \overline{CC}^t \to \varprojlim_p \overline{CC}^t / F^p\overline{CC}^t
\]
are isomorphisms.  (The former is clear since $F^0\overline{CC} = \overline{CC}$, whilst the latter uses the fact that $\overline{CC}$ is defined as a \emph{product} over $r$, rather than a sum.)  Similarly our filtration on $\calA_{\vv}$ is $I$-complete and $P$-complete in each degree.  The Eilenberg--Moore comparison theorem \cite[Theorem 7.4]{EilenbergMoore} then tells us that a filtered chain map $\overline{CC} \to \calA_{\vv}$ which induces an isomorphism on some page of the associated spectral sequences is automatically a quasi-isomorphism---even though the spectral sequences need not converge.  By \cref{lemgrPqis}, $P_{\vv}$ induces just such an isomorphism on the first page.
\end{proof}

\begin{rmk}
We learnt the Eilenberg--Moore comparison from Sheridan, who used it in \cite[Section 6.4]{SheridanFano} to compute certain Hochschild cohomology groups using the length filtration.
\end{rmk}

\begin{rmk}
Since we are in a relatively simple setting, one can prove the relevant special case of the comparison theorem by hand: consider the mapping cone $C$ of $P_{\vv}$; this inherits a filtration, and using the fact that $\gr P_{\vv}$ is a quasi-isomorphism one can show that each $C/F^pC$ is acyclic by induction on $p$; take inverse limits and use the Milnor sequence to obtain the theorem.
\end{rmk}


\subsection{Identifying the result}
\label{sscIdentifyingHH}

Having computed $\HH^*(\calA)$ as $\rH^*(\calA_{\vv})$, it remains to describe the latter and prove \cref{thmHHA}.  To simplify notation define a differential $\diff_{\vv}$ and product $\cdot$ on $\calA_{\vv}$ from the operations $\muvv^1$ and $\muvv^2$ by the familiar sign rule \eqref{eqdgAinfinity}.  These almost make $\calA_{\vv}$ into a dg-algebra; the only problem is that $\cdot$ need not be associative, although of course it \emph{is} associative up to homotopy.

Now let $\Po$ be the disc potential $\sum_{k\geq 2} \mu^k(\vv, \dots, \vv)$ of $\calA$ as usual.  Recall that $v_1, \dots, v_n$ is a basis for $V$, and $x_1, \dots, x_n$ are the dual formal variables in $R$, so that $\vv = \sum x_iv_i$.  For each $i$ we have
\[
\diff_{\vv} v_i = -\muvv^1(v_i) = - \sum_{k,l} \mu^k(\vv, \dots, \vv, v_i, \underbrace{\vv, \dots, \vv}_l) = - \frac{\partial}{\partial x_i} \sum_k \mu^k(\vv, \dots, \vv) = - \frac{\partial \Po}{\partial x_i} = - \diff \Po \contract v_i.
\]
Similarly, for all $i$ and $j$ we have
\[
v_i \cdot v_j + v_j \cdot v_i = - (\muvv^2(v_i, v_j) + \muvv^2(v_j, v_i)) = - \frac{\partial^2\Po}{\partial x_i \, \partial x_j} \quad \text{and} \quad v_i \cdot v_i = - \muvv^2(v_i, v_i) = - \frac{1}{2} \frac{\partial^2 \Po}{\partial x_i^2}.
\]
The division by $2$ makes sense here, even if $2$ is not invertible in $R_0$, since differentiating a monomial $x_1^{m_1}\dots x_n^{m_n}$ twice with respect to $x_i$ brings down the coefficient $m_i(m_i-1)$, which is even; the operator $\tfrac{1}{2}\partial_i^2$ operator is defined to bring down $m_i(m_i-1)/2$ instead.  Let $\tfrac{1}{2}\Hess(\Po)$ denote half the Hessian quadratic form on $V_R = R \otimes V$, defined by
\[
\frac{1}{2} \Hess(\Po) \Big(\sum_i a_i v_i\Big) = \sum_i a_i^2 \frac{1}{2}\frac{\partial^2 \Po}{\partial x_i^2} + \sum_{i < j} a_ia_j \frac{\partial^2 \Po}{\partial x_i \, \partial x_j}.
\]

The above calculations suggest that $\calA_{\vv}$ is related to the Clifford algebra $\Cl(-\tfrac{1}{2}\Hess(\Po))$ as defined in \eqref{eqCliffordAlgebra}---i.e.~the tensor algebra over $R$, on the free module $V_R$, modulo the two-sided ideal generated by the relations $v \otimes v = -\tfrac{1}{2}\Hess(\Po)(v)$ for all $v$---equipped with the differential $-\diff \Po \contract \bullet$.  Note that this differential extends uniquely from $V$ to the tensor algebra on $V_R$ using $R$-linearity and the Leibniz rule, and then descends to the Clifford algebra since it annihilates the relations.  We denote the dg-algebra $(\Cl(-\tfrac{1}{2}\Hess(\Po)), -\diff\Po \contract \bullet)$ by $\mathcal{C}$.  We cannot expect $\calA_{\vv}$ to be $A_\infty$-quasi-isomorphic to $\mathcal{C}$ since in general $\calA_{\vv}$ may have non-vanishing higher $A_\infty$-operations, but our main result is

\begin{thm}
\label{thmHAv}
There is a canonical isomorphism of unital associative $R$-algebras $\rH^*(\calA_{\vv}) \cong \rH^*(\mathcal{C})$.
\end{thm}

\begin{proof}
We'll show that $\rH^*(\calA_{\vv})$ and $\rH^*(\mathcal{C})$ canonically embed into $\Cl(-\tfrac{1}{2}\Hess(\Po)) \otimes_R \Jac(\Po)$, where $\Jac(\Po)$ is the Jacobian algebra $R/(\partial_i \Po)$, and that these embeddings have the same image.

For the first part, we claim the following sequence of canonical isomorphisms and embeddings
\begin{equation}
\label{eqUCTembedding}
\rH^*(\calA_{\vv}) \cong \rH^*(\calA_{\vv}) \otimes_R \Jac(\Po) \hookrightarrow \rH^*(\calA_{\vv} \otimes_R \Jac(\Po)) \cong \Cl(-\tfrac{1}{2}\Hess(\Po)) \otimes_R \Jac(\Po),
\end{equation}
and similarly for $\mathcal{C}$.  The middle inclusion map comes from the universal coefficient theorem, whilst the first isomorphism follows from the fact that $\partial_i \Po = -\diff_{\vv}v_i$ for each $i$, so $\partial_i \Po$ acts as $0$ on $\rH^*(\calA_{\vv})$. Analogous arguments apply to $\mathcal{C}$.  We just need to justify the second isomorphism, and for $\mathcal{C}$ this holds simply because the differential $-\diff\Po \contract \bullet$ vanishes after tensoring with $\Jac(\Po)$.

To complete the first part it remains to deal with the second isomorphism for $\calA_{\vv}$ (i.e.~the second isomorphism in \eqref{eqUCTembedding}).  To do this, for each $r$-tuple $I = (i_1 < \dots < i_r)$ in $\{1, \dots, n\}$ let $v_I$ denote the element of $\calA_{\vv}$ given by $((v_{i_1}\cdot v_{i_2})\dots)\cdot v_{i_r}$.  Reducing mod $\m$, $(\calA_{\vv}, \diff, \cdot)$ becomes the exterior algebra $(E, 0, \wedge)$, and the $v_I$ become the standard $R_0$-basis.  Before reducing, the $v_I$ therefore form an $R$-basis for $\calA_{\vv}$.  By the Leibniz rule, for all $I = (i_1 < \dots < i_r)$ we have
\[
\diff_{\vv} v_I = \sum_{j=1}^r (-1)^{j-1} (\diff v_{i_j}) v_{I \setminus \{i_j\}} = \sum_{j=1}^r (-1)^j \frac{\partial\Po}{\partial x_{i_j}} v_{I \setminus \{i_j\}}.
\]
In particular, $\diff_{\vv}$ vanishes modulo the ideal $(\partial_i \Po)$, so $\calA_{\vv} \otimes_R \Jac(\Po)$ has zero differential and is associative (consider the $3$-ary $A_\infty$-relation).  Moreover, the $v_I$ form a $\Jac(\Po)$-basis for it and the $v_i$ satisfy the relations
\[
\Big(\sum_i a_i v_i \Big) \cdot \Big(\sum_j a_j v_j\Big) = \frac{1}{2}\Hess(\Po) \Big(\sum_i a_i v_i\Big),
\]
so we conclude that it is canonically isomorphic to $\Cl(-\tfrac{1}{2}\Hess(\Po)) \otimes_R \Jac(\Po)$.

We thus have the claimed embeddings $\rH^*(\calA_{\vv}), \ \rH^*(\mathcal{C}) \hookrightarrow \Cl(-\tfrac{1}{2}\Hess(\Po)) \otimes_R \Jac(\Po)$ and we're left to identify their images.  For this, for each $I$ let $\overline{v}_I$ denote the element $v_{i_1}\dots v_{i_r}$ in $\mathcal{C}$.  There is no need for brackets here since $\mathcal{C}$ is already associative.  These $\overline{v}_I$ form an $R$-basis for $\mathcal{C}$, and the Leibniz rule gives
\[
-\diff \Po \contract \overline{v}_I = \sum_{j=1}^r (-1)^j \frac{\partial \Po}{\partial x_{i_j}} \overline{v}_{I\setminus\{i_j\}},
\]
so the $R$-linear map $\theta : \calA_{\vv} \to \mathcal{C}$ given by $v_I \mapsto \overline{v}_I$ is an isomorphism of chain complexes (ignoring the algebra structures).  The diagram
\[
\begin{tikzcd}
\rH^*(\calA_{\vv}) \arrow{rr}{\rH^*(\theta)} \arrow[hook]{dr} & & \rH^*(\mathcal{C}) \arrow[hook']{dl}
\\ & \Cl(-\tfrac{1}{2}\Hess(\Po)) \otimes_R \Jac(\Po)
\end{tikzcd}
\]
commutes and since the horizontal arrow is an isomorphism of $R$-modules we get the result.
\end{proof}

\begin{rmk}
Reducing the embedding of $\calA_{\vv}$ modulo $\m$ recovers the isomorphism $\rH^*(\calA) \cong \Cl(Q)$, where $Q$ is the form on $V$ given by the quadratic part of $-\Po$.
\end{rmk}

Combining \cref{thmHHisom} with \cref{thmHAv} we arrive at

\begin{cor}
\label{corHHA}
There is a canonical unital $R_0$-algebra isomorphism
\begin{flalign*}
&& \HH^*(\calA) &\cong \rH^*(\Cl(-\tfrac{1}{2}\Hess(\Po)), -\diff\Po \contract \bullet). && \qed
\end{flalign*}
\end{cor}

\begin{ex}
Suppose that $R_0$ is a field of characteristic $2$, that $V$ is the one-dimensional $R_0$-vector space $\langle v\rangle$, and that $\calA = R_0[v]/(v^2+1)$, with higher $A_\infty$-operations $\mu^k$ satisfying $\mu^k=0$ whenever $k$ is odd (here $v$ lies in degree $1$ mod $2$).  This has
\[
\Po = x^2 + c_4x^4 + c_6x^6 + \dots
\]
for some $c_4, c_6, \dots$ in $R_0$, so $\diff\Po = 0$ and
\[
\HH^*(\calA) \cong R_0 \llbracket x \rrbracket [v]/(v^2+1+c_6 x^4 + c_{10} x^8 + \dots).\qedhere
\]
\end{ex}

We already know that $\rH^*(\calA_{\vv})$ is graded-commutative, since it is isomorphic to $\HH^*(\calA)$, but we actually have

\begin{prop}
\label{propCentre}
The image of the embedding $\rH^*(\calA_{\vv}) \hookrightarrow \Cl(-\tfrac{1}{2}\Hess(\Po)) \otimes_R \Jac(\Po)$, or equivalently the image of $\rH^*(\mathcal{C})$, is contained in the centre of the codomain.
\end{prop}
\begin{proof}
We'll work with the image of $\rH^*(\mathcal{C})$.  With notation as in the proof \cref{thmHAv}, suppose $a \coloneqq \sum_I a_I \overline{v}_I$ is a cocycle in $\mathcal{C}$.  Since the differential $-\diff \Po \contract \bullet$ decreases the length $|I|$ of $I$ by $1$, we may assume that the set $\{|I|: a_I \neq 0\}$ has a single element $r$, and we claim that for all $l$ we have $v_la = (-1)^r av_l$ modulo $(\partial_i \Po)$.  Then $a$ graded-commutes with each $v_l$ in $\Cl(-\tfrac{1}{2}\Hess(\Po)) \otimes_R \Jac(\Po)$, and hence lies in the centre.

For each $I=(i_1 < \dots < i_r)$, each $l \in \{1, \dots, n\}$, and each $s \in \{1, \dots, r\}$, consider
\[
\delta_I^{l,s} \coloneqq a_I v_{i_1} \dots v_{i_{s-1}} v_l v_{i_s} \dots v_{i_r} + a_I v_{i_1} \dots v_{i_s} v_l v_{i_{s+1}} \dots v_{i_r}.
\]
Since the algebra underlying $\mathcal{C}$ is $\Cl(-\tfrac{1}{2}\Hess(\Po))$, we have
\[
\delta_I^{l,s} = \frac{\partial^2 \Po}{\partial x_l \, \partial x_{i_s}} a_I v_{i_1} \dots \widehat{v}_{i_s} \dots v_{i_r} = \frac{\partial}{\partial x_l} \Big(\frac{\partial \Po}{\partial x_{i_s}} a_I v_{i_1} \dots \widehat{v}_{i_s} \dots v_{i_r}\Big) - \frac{\partial \Po}{\partial x_{i_s}} \frac{\partial a_I}{\partial x_l} v_{i_1} \dots \widehat{v}_{i_s} \dots v_{i_r},
\]
where $\widehat{\phantom{v}}$ denotes omission.  The final term on the right-hand side vanishes modulo $(\partial_i \Po)$, so we get
\[
v_la - (-1)^r av_l = \sum_I \sum_{s=1}^r (-1)^{s-1} \delta_I^{l,s} = \frac{\partial}{\partial x_l} \sum_I \sum_{s=1}^r (-1)^{s-1}\frac{\partial \Po}{\partial x_{i_s}} a_I v_{i_1} \dots \widehat{v}_{i_s} \dots v_{i_r} \mod (\partial_i \Po).
\]
The double sum on the right-hand side is exactly $\diff\Po \contract a$, and this vanishes by the assumption that $a$ is a cocycle, proving the claim.
\end{proof}


\section{Monotone Lagrangian tori}
\label{secMonotoneTori}

In this short final section we change direction slightly, and move from pure algebra into Floer theory, with which some familiarity is assumed.  We begin by discussing a weakening of the definition of a superfiltered $A_\infty$-deformation of $E$, before recapping the Floer theory of monotone tori, and then going on to state and prove our main results.


\subsection{Weak superfiltered deformations}

In this section we shall consider an $A_\infty$-algebra $\calA$ arising from the Floer theory of a monotone Lagrangian torus $L$.  This $A_\infty$-algebra is naturally superfiltered, minimal, and cohomologically unital.  It also comes with a natural identification of the algebra $\gr \rH^*(\calA)$ with $E = \Lambda V$, where $V = \rH^1(L; R_0)$.  Moreover, the latter algebra identification extends non-canonically to an $A_\infty$-isomorphism between $\gr \calA$ and $E$.  Such an $\calA$ is thus almost a superfiltered $A_\infty$-deformation of $E$.  Motivated by this we introduce the following definition.

\begin{defn}
A \emph{weak superfiltered $A_\infty$-deformation of $E$} is a superfiltered $A_\infty$-algebra $\calA$, which is cohomologically unital and minimal, and which is equipped with an algebra isomorphism $\phi : \gr \rH^*(\calA) \to E$ that extends to a ($\Z$-graded) $A_\infty$-isomorphism $\Phi : \gr \calA \to E$ ($\Phi$ should extend $\phi$ in the sense that $\Phi^1 = \phi$).  The map $\phi$, but not $\Phi$, is part of the data.
\end{defn}

We will refer to our previous definition as a \emph{strong} superfiltered $A_\infty$-deformation if we wish to distinguish it from this new version.

The goal of the present subsection is to prove the following result.

\begin{lem}
Given a weak superfiltered $A_\infty$-deformation $\calA$ of $E$, there exists a strong superfiltered $A_\infty$-deformation $\calA_\mathrm{str}$ and an isomorphism $\Sigma : \calA \to \calA_\mathrm{str}$ of cohomologically unital superfiltered $A_\infty$-algebras, such that $\gr \Sigma^1$ intertwines the identifications of $\gr \rH^*(\calA)$ and $\gr \rH^*(\calA_\mathrm{str})$ with $E$.  Such an $\calA_\mathrm{str}$, which we call a \emph{strengthening} of $\calA$, is canonical up to $1$-equivalence.
\end{lem}
\begin{proof}
Fix a choice of $A_\infty$-isomorphism $\Phi : \gr \calA \to E$ extending the given algebra isomorphism $\phi : \gr \rH^*(\calA) \to E$, and lift $\Phi^1 = \phi$ to an identification of the module underlying $\calA$ with $E$, as in \cref{exAinftyDef}.  In this way we think of $\calA$ as a cohomologically unital and minimal superfiltered $A_\infty$-structure on $E$, such that the leading term of $\mu^2_\calA$ coincides with $\mu^2_E$.  We can likewise think of $\Phi$ as a ($\Z$-graded) formal diffeomorphism of $E$ such that $\Phi^1 = \id$ and such that the leading terms of the operations on $\Phi_* \calA$ agree with those on $E$.

This $\Phi_* \calA$ is almost what we want: it's a minimal superfiltered $A_\infty$-structure on $E$ such that the leading terms of the operations coincide with the standard operations on $E$.  The only problem is that it's cohomologically unital, rather than strictly unital.  Recall, however, that we can modify the operations by another formal diffeomorphism, say $\Psi$, to achieve strict unitality.  Moreover, this can be done in a way which respects the $\Z/2$-grading and filtration---for an explicit recipe see Seidel \cite[Lemma 2.1]{SeidelBook}.  Since the leading terms were already strictly unital they are unaffected.  Let $\calA_\mathrm{str} = \Psi_*\Phi_* \calA$ be the result of this procedure.  By construction this is a strong superfiltered $A_\infty$-deformation of $E$, and the map $\Psi \circ \Phi$ is the required $\Sigma : \calA \to \calA_\mathrm{str}$.

Now suppose that $\Sigma' : \calA \to \calA_\mathrm{str}'$ is another strengthening of $\calA$.  The map $\Sigma' \circ \Sigma^{-1} : \calA_\mathrm{str} \to \calA_\mathrm{str}'$ is almost a $1$-equivalence; the only problem is that it's only cohomologically unital.  But by a similar procedure to the modification of $\Phi_*\calA$ we may make it strictly unital in a way that preserves its other properties.
\end{proof}

\begin{defn}
\label{defWeakPotential}
The \emph{disc potential} of a weak superfiltered deformation $\calA$ of $E$ is the disc potential of any strengthening.  By the previous result this is well-defined up to $1$-equivalence.
\end{defn}


\subsection{Floer algebras of monotone tori}
\label{sscFloerAlgebras}

For the remainder of the paper, let $L$ denote a monotone Lagrangian $n$-torus inside a symplectic manifold $(X, \omega)$ which is compact or tame at infinity.  Associated to $L$ is its (super)potential $W_L : \rH^1(L; R_0^\times) \to R_0$, defined on the space of rank $1$ local systems on $L$.  This sends a local system $\rho \in \rH^1(L; R_0^\times)$ to the count of rigid holomorphic discs bounded by $L$, each weighted by the holonomy of $\rho$ around its boundary.  A precise description is given in \cite[Definition 2.2]{ChoHongLauTorus}; we set their variable $T$ to be $1$.  Note that $W_L$ can be viewed as an element of the group algebra $R' \coloneqq R_0 [ \rH_1(L; \Z) ]$.

For each local system $\rho$, we have an object $(L, \rho)$ in the $\lambda$-summand $\Fuk_\lambda(X, \omega)$ of the monotone Fukaya category of $X$, where $\lambda = W_L(\rho)$.  See \cite{SheridanFano} for the detailed construction of this category.  The endomorphism algebra of $(L, \rho)$ is the Floer algebra $CF^*((L, \rho), (L, \rho))$, which is a cohomologically unital $\Z/2$-graded $A_\infty$-algebra over $R_0$.  Roughly, this is a deformation of the ordinary cohomology algebra of $L$, obtained by adding quantum corrections from holomorphic discs.  Monotonicity ensures that all quantum corrections are degree-decreasing, so the algebra is naturally superfiltered.  Moreover, if $\rho$ is a critical point of $W_L$ then the Floer cohomology $\HF^*((L, \rho), (L, \rho))$ can be identified with $\rH^*(L; R_0)$, and this becomes a canonical algebra isomorphism at the associated graded level.

From now on, assume $\rho$ is a critical point of $W_L$ with critical value $\lambda$.  The precise Floer algebra $CF^*((L, \rho), (L, \rho))$ depends on a choice of model and auxiliary data.  We shall use a \emph{pearl model}, as outlined in \cite{ChoHongLauTorus} following Biran--Cornea \cite{biran2007quantum}.  This requires a choice of Morse data on $L$ as well as various perturbation data.  We shall take a perfect Morse function, and place more constraints on it later.  The resulting $CF^*((L, \rho), (L, \rho))$, which we denote by $\calA$, is minimal, and the leading terms of its $A_\infty$-operations describe the $A_\infty$-structure on $\rH^*(L; R_0)$.  Tori are well-known to be formal, meaning that the $A_\infty$-structure on $\rH^*(L; R_0)$ is isomorphic to the one with vanishing higher operations.  Therefore $\calA$ is precisely a weak superfiltered $A_\infty$-deformation of $E = \Lambda V$, where $V = \rH^1(L; R_0)$.

\begin{warn}
\label{warnOrientationScheme}
The signs appearing in $\calA$ depend on a choice of orientation scheme for the Floer-theoretic moduli spaces.  This is separate from the choice of $A_\infty$-algebra sign conventions: for example, if one author uses an orientation scheme which produces operations $\mu^k$ satisfying our standard $A_\infty$-relations \eqref{eqAinfinityRelations}, then another may twist the orientation scheme to give operations $\widetilde{\mu}^k \coloneqq (-1)^k \mu^k$, and these still satisfy the same relations.  Our identification of $\gr \rH^*(\calA)$ with $E$ fixes the orientations of the moduli spaces defining $\mu^2$---we need the leading term of $(-1)^{\lvert a_1 \rvert} \mu^2(a_2, a_1)$ to coincide with $a_2 \wedge a_1$---but we are free to twist the orientations of the moduli spaces defining $\mu^1$.  Our results will all therefore depend on an unknown sign $\eps \in \{\pm1\}$ that arises from this $\mu^1$ ambiguity.  (There may be additional ambiguity in the higher operations, but it does not affect our results.)
\end{warn}


\subsection{Statement of results}

In \cite{ChoHongLauTorus} Cho--Hong--Lau introduced a geometric version of the localised mirror functor.  Stated precisely, the output of their construction is
\begin{thm}[{\cite[Theorem 1.1]{ChoHongLauTorus}}]
There is a geometrically-defined $A_\infty$-functor
\begin{equation}
\label{eqLMFtorus}
\pushQED{\qed}
\LMF{L}_\mathrm{geom} : \Fuk_\lambda(X, \omega) \to \mf(R', W_L - \lambda).\qedhere
\popQED
\end{equation}
\end{thm}
This functor is morally equivalent to the localised mirror functor described in \cref{sscLocalisedMirror}, but replaces the algebraic description in terms of insertions of $\vv$ into the $A_\infty$-operations with a geometric description which takes the ordinary $A_\infty$-operations on the Fukaya category and modifies the definition by incorporating additional weights in $R'$.  One should think of the insertions of $\vv$ as giving the formal expansions of these weights.

\begin{rmk}
In the construction of \eqref{eqLMFtorus} the assumption of monotonicity can be weakened to `positivity'---see \cite[Assumption 2.1]{ChoHongLauTorus}---at the expense of working over a Novikov ring and (a priori at least) making the construction dependent on a specific choice of almost complex structure on $X$.  As Cho--Hong--Lau remark, if one is happy to employ more sophisticated techniques then it should be possible in characteristic $0$ to weaken the assumption further, to unobstructedness of $L$.  In this case, however, one loses the filtration that is crucial to our arguments.
\end{rmk}

The category $\mf(R', W_L - \lambda)$ contains a superfiltered object $\mfE_0'$ corresponding to the skyscraper sheaf at $\rho$, and we denote the image of $(L, \rho)$ under $\LMF{(L, \rho)}_\mathrm{geom}$ by $\mfE'$.  Cho--Hong--Lau \cite[Theorems 9.1 and 9.4]{ChoHongLauTorus} showed (for $R_0 = \C$) that if $L$ is a monotone toric fibre of dimension at most $4$ then $\mfE'$ is isomorphic to $\mfE_0'$; moreover this isomorphism is via a `quantum change of variables'.  They conjectured \cite[Section 8]{ChoHongLauTorus} that such a quantum change of variables exists and provides an isomorphism between $\mfE'$ and $\mfE_0'$ for all monotone tori $L$.  Using our technique---keep track of filtrations and then use a spectral sequence---we prove this (\cref{propQuantumChangeOfCoordinates}) and obtain
\begin{mthm}[\cref{thmLocalMSText}]
\label{thmLocalMS}
The minimal model we construct for $\eend_{\mffilt}(\mfE_0')$ is a strengthening of $\calA$.
\end{mthm}
Expanding out the jargon, this means that the Floer algebra of $(L, \rho)$ is quasi-isomorphic as a cohomologically unital superfiltered $A_\infty$-algebra to the endomorphism algebra of $\mfE_0'$, in such a way that the induced isomorphism $\gr \HF^*((L, \rho), (L, \rho)) \cong \gr \rH^*(\eend(\mfE_0'))$ is compatible the identifications of both sides with $E$.  This result may be called \emph{local mirror symmetry}, since it matches a formal neighbourhood of $\rho$ in $\mf(R', W_L - \lambda)$ with the piece of $\Fuk_\lambda (X, \omega)$ split-generated by $(L, \rho)$.

In general, computing the $A_\infty$-structure on the Floer algebra is very difficult, and to the best of the author's knowledge the only previously known cases for monotone tori are the low-dimensional monotone toric fibres covered by Cho--Hong--Lau, and cases where the algebra is intrinsically formal (i.e.~\emph{any} $A_\infty$-structure on the underlying Clifford algebra is formal, meaning quasi-isomorphic to the one with vanishing higher operations).  Using a generation result of Evans--Lekili \cite[Corollary 1.3.1]{EvansLekiliGeneration}, Cho--Hong-Lau \cite[Corollary 1.3]{ChoHongLauTorus} proved global mirror symmetry for all compact toric Fano manifolds, but this does not directly give the full $A_\infty$-structure.

It is a folklore result that the disc potential of $CF^*((L, \rho), (L, \rho))$ is in some sense the same as $W_L - \lambda$, and our final result makes this precise.  This is a straightforward consequence of \cref{thmLocalMS} after relating the right-hand side of \eqref{eqLMFtorus} to the previously-appearing $\mf^\mathrm{(filt)}(R, w)$, which we do as follows.  Recall that we fixed a basis $v_1, \dots, v_n$ for $V$, so that $R = R_0\llbracket x_1, \dots, x_n \rrbracket$ where the $x_i$ are the dual coordinates on $V$.  Assume that the basis $v_1, \dots, v_n$ is induced from $\rH^1(L; \Z)$ and let $z_i$ be the corresponding coordinates in $R_0[\rH^1(L; \Z)] = R_0[z_1^{\pm 1}, \dots, z_n^{\pm n}]$.  Identifying $z_i$ with $\rho_i(1+\eps x_i)$, where $\rho_i$ is the $z_i$-coordinate of $\rho$ and $\eps$ is the sign from \cref{warnOrientationScheme}, we can view $R$ as the completion of $R_0[\rH^1(L; \Z)]$ at $\rho$, and take $w$ to be the expansion of $W_L - \lambda$ about this point.  Our assumption that $\lambda = W_L(\rho)$ and that $\rho$ is a critical point of $W_L$ ensures $w$ lies in $\m^2 \subset R$.

The statement is then

\begin{mthm}[\cref{corPoComputation}]
\label{thmPoComputation}
The disc potential of $\calA$ in the sense of \cref{defWeakPotential} is $1$-equivalent to the formal expansion of $W_L - \lambda$ about $\rho$, under the identification $z_i = \rho_i(1+\eps x_i)$ described above.
\end{mthm}

\begin{rmk}
This identification depends on the choice of basis for $\rH_1(L; \Z)$, but that choice doesn't affect the $1$-equivalence class of the expansion.  Ignoring the sign $\eps$, there is another `obvious' identification we could try, namely $z_i = \rho_i + x_i$, but this would give the wrong answer.
\end{rmk}

\Cref{thmPoComputation} extends previous results of Cho \cite{ChoProducts} and Fukaya--Oh--Ohta--Ono \cite[Theorem 4.5]{FOOOToricI}, which deal with the case where $L$ is a toric fibre and the ground ring is a field of characteristic zero (the latter paper goes far beyond the monotone setting), and of Biran--Cornea \cite[Section 3.3]{BiranCorneaLagrangianTopology}, who computed the quadratic part for general monotone $L$.  It is interesting to note that our proof involves no Floer theory beyond the construction and basic properties of the Fukaya category and the localised mirror functor.


\subsection{The matrix factorisations}

The functor $\LMF{L}_\mathrm{geom}$ sends $(L, \rho)$ to a superfiltered matrix factorisation $\mfE'$ and provides a cohomologically unital superfiltered $A_\infty$-algebra map
\[
\Phi' : \calA \to \calB' \coloneqq \eend_{\mffilt}(\mfE').
\]
Using our pearl model, the underlying module of $\mfE'$ is $E_{R'} \coloneqq R' \otimes E$.  Our main task is to identify the leading term of the squifferential $\diff_{\mfE'}$ and of $(\Phi')^1$.

\begin{rmk}
Cho--Hong--Lau used `leading order term' in \cite{ChoHongLauTorus} to mean the next term down in the filtration; what we call the leading term they called the classical part.
\end{rmk}

First we must choose the `gauge hypertori' $H_i$.  To do this, begin by fixing a diffeomorphic identification $L \cong (S^1)^n$ such that $\langle v_i, \gamma_j\rangle = \delta_{ij}$, where $\gamma_j$ is the loop that goes once positively (anticlockwise) around the $j$th $S^1$ factor.  Now define $H_i$ to be the hypertorus $(S^1)^{i-1} \times \{p\} \times (S^1)^{n-i}$, co-oriented by the positive orientation of $S^1$, where $p$ is an arbitrarily chosen point in $S^1$.  The crucial property of $H_i$ is that intersecting a $1$-cycle with it corresponds to pairing the cycle with $v_i$.

Next we choose a Morse function $f$ on $L$.  For this, let $f_{S^1}$ be a perfect Morse function on $S^1$ with min at $q_0$ and max at $q_1$, where $p$, $q_0$, and $q_1$ are in clockwise order.  To determine signs we must orient the descending manifolds.  For $q_0$ we take the canonical orientation, whilst for $q_1$ we take the anticlockwise orientation.  Now define $f$ to be the product of $n$ copies of $f_{S^1}$, using our identification $L \cong (S^1)^n$.  The module underlying $\calA$ is the Morse cochain complex of $f$, which is the $n$th tensor power of the Morse cochain complex of $f_{S^1}$.  Given $I = (I_1, \dots, I_n) \in \{0, 1\}^n$ we denote by $q_I$ both the basis element $q_{I_1} \otimes \dots \otimes q_{I_n}$ of $\calA$ and the corresponding critical point $q_{I_1} \times \dots \times q_{I_n}$ of $f$.  The element $v_i$ is represented by $q_{e_i}$, where $e_i$ is the $i$th standard basis vector.  The Morse complex of $f$ (over $R_0$) is naturally identified with $\rH^*(L; R_0)$, and this gives our identification between $\gr \rH^*(\calA)$ and $E$.

The leading term of the squifferential on $\mfE'$ counts rigid Morse flowlines on $L$, with each flowline $\gamma$ weighted by a factor of
\begin{equation}
\label{eqWeightFactor}
\prod_{i=1}^n \Big(\frac{z_i}{\rho_i}\Big)^{H_i \cdot \gamma},
\end{equation}
where $H_i \cdot \gamma$ is the intersection number of $\gamma$ with $H_i$.  By analogy with \cref{sscLocalisedMirror}, the sign attached to a flowline contributing to $\diff_{\mfE'}a$ is $(-1)^{\lvert a\rvert}$ times the sign with which that flowline contributes to $\mu^1_\calA(a)$.  This in turn is the sign with which the flowline contributes to the differential on $\calA$ via \eqref{eqdgAinfinity}.

Similarly, for all $I$ and $J$, the leading term of $(\Phi')^1(q_J)(q_I)$ counts rigid perturbed Y-shaped Morse flow trees with inputs $q_I$ and $q_J$.  Each such tree is weighted by \eqref{eqWeightFactor}, where $\gamma$ is the path through the tree from the first input to the output, as shown in the right-hand part of \cref{figMorse}.  The sign is $(-1)^{\lvert q_I \rvert}$ times the sign with which the tree contributes to $\mu^2_\calA(q_J, q_I)$.  This is simply the sign with which the tree contributes to $q_J \wedge q_I$.
\begin{figure}[ht]
\begin{tikzpicture}

\def\r{2cm}
\def\blobsize{1.5mm}
\draw (0, 0) circle[radius=\r];
\draw (\r, 0) node[blob]{};
\draw (0, \r) node[blob]{};
\draw (0, -\r) node[blob]{};
\draw (\r, 0) node[anchor=west]{$p$};
\draw (0, -\r) node[anchor=north]{$q_0$};
\draw (0, \r) node[anchor=south]{$q_1$};

\draw (250:\r) node[blob](q0p){};
\draw (q0p) node[anchor=north]{$q_0'$};
\draw (110:\r) node[blob](q1p){};
\draw (q1p) node[anchor=south]{$q_1'$};

\draw[line width=0.5mm, opacity=0.3, decoration={markings, mark=at position 0.5 with {\arrow{>}}}, postaction=decorate] (-100:1.7cm) arc [radius =1.7cm, start angle=-100, end angle=-260];
\draw[line width=0.5mm, opacity=0.3, decoration={markings, mark=at position 0.5 with {\arrow{>}}}, postaction=decorate] (-80:1.7cm) arc [radius =1.7cm, start angle=-80, end angle=80];
\draw (-1.3cm, -0.1) node{$\gamma_+$};
\draw (1.3cm, -0.1) node{$\gamma_-$};

\begin{scope}[xshift=7cm]
\draw (-1, -\r) node[blob](qI){};
\draw (1, -\r) node[blob](qJ){};
\draw (0, \r) node[blob](out){};

\draw[snakeit, ->] (qI) -- (-0.5, -\r/2);
\draw[snakeit] (-0.5, -\r/2) -- (0, 0);
\draw[snakeit, ->] (qJ) -- (0.5, -\r/2);
\draw[snakeit] (0.5, -\r/2) -- (0, 0);
\draw[snakeit, ->] (0, 0) -- (0, \r/2);
\draw[snakeit] (0, \r/2) -- (out);

\draw (qI) node[anchor=north]{$q_I$};
\draw (qJ) node[anchor=north]{$q_J$};
\draw (out) node[anchor=south]{output};
\end{scope}

\begin{scope}[xshift=6.5cm]
\draw[rounded corners=0.5cm, decoration={markings, mark=at position 0.55 with {\arrow{>}}}, postaction=decorate] (-1, -\r) -- (0, 0) -- (0, \r);
\draw (-0.35, 0) node{$\gamma$};
\end{scope}

\end{tikzpicture}
\captionsetup{width=\linewidth}
\caption{Critical points and flowlines on $S^1$, and the trees computing the leading term of $(\Phi')^1(q_J)(q_I)$.\label{figMorse}}
\end{figure}

\begin{lem}
\label{lemSquifferentialLead}
The leading term of the squifferential on $\mfE'$ is $-\eps \vv' \wedge \bullet$, where $\vv' = \sum_i (z_i/\rho_i - 1) v_i$ and $\eps \in \{\pm 1\}$ is an unknown sign depending on the choice of orientation scheme.  (This defines the $\eps$ appearing in \cref{warnOrientationScheme}.)
\end{lem}
\begin{proof}
The rigid Morse flowlines are constant on $n-1$ of the $S^1$ factors, and flow from $q_0$ to $q_1$ on the other factor---say the $i$th one.  Restricting to this $i$th factor, there are two flowlines from $q_0$ to $q_1$, which we denote by $\gamma_\pm$ as shown in \cref{figMorse}.  The path $\gamma_-$ passes through $p$ positively (anticlockwise) so has weight $z_i/\rho_i$ according to \eqref{eqWeightFactor}, whilst $\gamma_+$ avoids $p$ so has weight $1$.  These paths contribute with signs $\pm \eps$ respectively, for some $\eps \in \{\pm 1\}$ (indepdendent of $i$), so the leading term of the squifferential on this factor looks like $-\eps(z_i/\rho_i-1)v_i \wedge \bullet$.  Passing back to $L$ itself, the Leibniz rule tells us that total leading term of the squifferential is $-\eps \vv' \wedge \bullet$.
\end{proof}

\begin{lem}
We can choose the perturbations for our Y-shaped trees so that the leading term of $(\Phi')^1(v_i)$ is $v_i \wedge \bullet$.
\end{lem}
\begin{proof}
We use the usual Morse function $f$ on the output leg and first input leg of the $Y$.  On the second input leg we use a deformation of $f$ which corresponds to perturbing $q_0$ and $q_1$ slightly to points $q_0'$ and $q_1'$.  The only requirement we make is that $q_1'$ lies slightly anticlockwise of $q_1$, as shown in \cref{figMorse}.  Now fix $J \in \{0, 1\}^n$ and consider the trees computing the leading term of $(\Phi')^1(v_i)(q_J)$.  These are precisely the trees computing $v_i \wedge q_J$, counting with the same signs, but weighted according to \eqref{eqWeightFactor}.  It remains to show that with our perturbations these weights are all $1$.

On the $i$th factor the tree must have inputs $q_0$ and $q_1'$ and output $q_1$, so the `$\gamma$-path' determining the weight is roughly $\gamma_-$.  In particular, it avoids $p$.  On each of the other factors the tree either has inputs $q_0$ and $q_0'$ and output $q_0$, or has inputs $q_1$ and $q_0'$ and output $q_1$.  In these cases the $\gamma$-path is the constant path at $q_0$ or at $q_1$ respectively, so again avoids $p$.  The upshot is that all of the trees avoid the gauge hypertori and hence are weighted by $1$ as needed.
\end{proof}

Now let $\m'$ denote the ideal of $R'$ generated by the $z_i/\rho_i-1$.  The fact that $\rho$ is a critical point of $W_L$ means that $W_L - \lambda$ lies in $(\m')^2$, and analogously to \cref{sscDefE0} we can pick $\ww'$ in $\m' \otimes V^\vee$ such that $\ww' \contract \vv' = W_L - \lambda$.  We can then define another superfiltered object $\mfE_0'$ in $\mf(R', W_L - \lambda)$ to have underlying module $R' \otimes E$ and squifferential
\[
\diff_{\mfE_0'} : a \mapsto -\eps(\vv' \wedge a + \ww' \contract a),
\]
where $\eps$ is as appears in \cref{lemSquifferentialLead} and \cref{warnOrientationScheme}.  The same argument as for \cref{propmfEqual,lemidERCocycle} shows

\begin{prop}
\label{propQuantumChangeOfCoordinates}
The identity map on the underlying modules can be corrected with lower order terms to a cocycle in $\hom^0_{\mffilt}(\mfE', \mfE_0')$.  Conjugation by this cocycle gives a superfiltered dg-algebra isomorphism $\Psi'$ from $\calB' = \eend_{\mffilt}(\mfE')$ to $\calB_0' \coloneqq \eend_{\mffilt}(\mfE_0')$.\hfill$\qed$
\end{prop}

\begin{rmk}
This conjugation provides the conjectured `quantum change of coordinates' mentioned above \cref{thmLocalMS}.
\end{rmk}

Composing this $\Psi'$ with the $\Phi' : \calA \to \calB$ given by $\LMF{L}_\mathrm{geom}$, we obtain a cohomologically unital superfiltered $A_\infty$-algebra map
\[
\Psi' \circ \Phi' : \calA \to \calB_0'
\]
such that the leading term of $(\Psi' \circ \Phi')^1 (v_i)$ is $v_i \wedge \bullet$.  By the same arguments as for $\calB_0$ in \cref{secSurj}, we have that $\rH^*(\calB_0')$ is isomorphic to $E$, such that the action of $e \in E$ on $\mfE'_0$ has leading term $e \wedge \bullet$.  Therefore $\Psi' \circ \Phi'$ intertwines the identifications of $\gr \rH^*(\calA)$ and $\gr \rH^*(\calB_0')$ with $E$.


\subsection{Completing the proofs}

By mimicking the construction of the minimal model $\calB_0^\mathrm{min}$ for $\calB_0$ from \cref{secSurj}, we construct a minimal model $(\calB_0')^\mathrm{min}$ for $\calB_0'$, along with maps $\iota'$, $\pi'$ and $\eta'$.  In fact, the argument for $\calB_0'$ is slightly simpler than that for $\calB_0$, because $\calB_0'$ is free as an $R_0$-module so we can avoid the decomposition into the $K^{s,q}$ that was needed in \cref{lemetam1}.  We do, however, choose to split $\calB_0'$ into subcomplexes $K^q \coloneqq R' \otimes \Lambda^q V \otimes \Lambda V^\vee$, to give us the control over $\eta'_{-1}$ needed in order to prove the following analogue of \cref{lemetaleading}.

\begin{lem}
\label{lemetapleading}
For $r \in \m'$, the leading term of $\eta'(r \id_{E_{R'}})$ is of the form $\sum_i m'_i(r) v_i^\vee$, where the $m'_i(r)$ are elements of $R'$ satisfying $\sum_i (z_i/\rho_i - 1) m'_i(r) = -\eps r$.
\end{lem}
\begin{proof}
The argument of \cref{lemetaleading} goes through with appropriate primes added to the notation and with the subcomplex $K^0$ in place of the separate pieces $K^{p,0}$.  At the end the homotopy condition gives $-\eps r$ instead of $-r$, because the leading term of $\mfE'_0$ is $-\eps \vv \wedge \bullet$ rather than $-\vv \wedge \bullet$ and hence the leading term of $\mu^1_{\calB_0'}$ is $\eps \mu^1_{E_\mathrm{dg}}$.
\end{proof}

Letting $\Pi' : \calB_0' \to (\calB_0')^\mathrm{min}$ be the projection, we obtain

\begin{thm}
\label{thmLocalMSText}
The map $\Pi' \circ \Psi' \circ \Phi' : \calA \to (\calB_0')^\mathrm{min}$ is a strengthening of $\calA$.\hfill$\qed$
\end{thm}

\begin{cor}
\label{corPoComputation}
The disc potential of $\calA$ is $1$-equivalent to the formal expansion of $W_L - \lambda$ about $\rho$ under the identification $z_i = \rho_i(1+\eps x_i)$.
\end{cor}
\begin{proof}
Setting $z_i = \rho_i(1+\eps x_i)$, the crucial equality satisfied by the $m_i'$ in \cref{lemetapleading} reduces to that satisfied by the $m_i$ in \cref{lemetaleading}.  We can now simply follow through the proof of \cref{propDiscPotential} with appropriate primes added, to obtain the result.
\end{proof}

As a simple application, combining this result with \cref{Theorem1} gives

\begin{cor}
\label{corFormality}
The model $\calA$ for $CF^*((L, \rho), (L, \rho))$ can be strengthened to the formal $A_\infty$-structure on its underlying Clifford algebra (meaning the $A_\infty$-structure with vanishing higher operations) if and only if $W_L - \lambda$ can be made homogeneous of degree $2$ by a formal change of variables about $\rho$ whose first order term is the identity.\hfill$\qed$
\end{cor}

\begin{ex}
In \cite[Proposition 1.3, Section 3.4]{TonkonogCO}, Tonkonog studied the case where $L$ is the equatorial Lagrangian torus (circle) on $S^2$ and $\rho$ is trivial.  He showed that the Floer algebra is non-formal in characteristic $2$ (precisely: the higher operations cannot be eliminated by a $1$-equivalence), using a direct Massey product computation and, independently, a more general algebraic argument.  It was previously known to be formal in all other characteristics.  \Cref{corFormality} gives a new perspective on this result: we have $W_L - \lambda = z+1/z-2$, which becomes
\begin{equation}
\label{eqExpandedPotential}
x^2-x^3+x^4-\dots
\end{equation}
under the substitution $z=1+x$, and so the Floer algebra is formal if and only if \eqref{eqExpandedPotential} can be transformed to $y^2$ by a formal change of variables $y=x+a_2x^2+a_3x^3+\dots$.  This happens if and only if \eqref{eqExpandedPotential} has a square root over $R_0$, which, by binomial expansion, is if and only if $2$ is invertible.
\end{ex}

\bibliography{ClassificationBib}
\bibliographystyle{utcapsor2}

\end{document}